\documentclass[12pt]{report}
\providecommand{\MR}{\relax\ifhmode\unskip\space\fi MR }

\providecommand{\href}[2]{#2}
\usepackage{amsmath, amsfonts, amssymb, amsthm}  
\usepackage{tocbibind}    
\newenvironment{sketch}[1][\proofname]{\proof[#1]\mbox{}\\*}{\endproof}
\usepackage{color}
\usepackage{blindtext}
\usepackage[all]{xy}
\usepackage{amsthm}   
\usepackage{amsmath, amsfonts, amssymb, amsthm}  
\usepackage[pdftex]{graphicx}
\usepackage[active]{srcltx}   
\usepackage{enumerate, booktabs, mathrsfs}  
\newtheorem{definition}{Definition}[chapter]

\newtheorem{remark}{Remark}[chapter]
\newtheorem{sketch of the proof}{Sketch of the proof}
\newtheorem{example}{Example}[chapter]
\newtheorem{theorem}{Theorem}[chapter]
\newtheorem{lemma}{Lemma}[chapter]
\newtheorem{proposition}{Proposition}[chapter]

\begin{document}  
\begin{titlepage}
\vfill
\begin{center} 
A CLASSIFICATION OF HOMOGENEOUS K\"{A}HLER MANIFOLDS WITH DISCRETE ISOTROPY 
AND TOP NON VANISHING HOMOLOGY IN CODIMENSION TWO
\vspace{32pt}

A Thesis \\
Submitted to the Faculty of Graduate Studies and Research \\
In Partial Fulfillment of the Requirements \\
For the Degree of \\
Doctor of Philosophy \\
In \\
Mathematics \\
University of Regina \\

\vspace{32pt}

By \\
Seyedruhallah Ahmadi \\
Regina, Saskatchewan \\
July, 2013 \\

\vspace{32pt}

\copyright \; Copyright 2013: S.Ruhallah Ahmadi \\

\end{center}
\vfill
\end{titlepage}
\clearpage  
\pagenumbering{roman}  
\chapter*{Abstract}
\addcontentsline{toc}{chapter}{Abstract}
Suppose $G$ is a connected complex Lie group and $\Gamma$ is a discrete subgroup 
such that $X := G/\Gamma$ is K\"{a}hler and   
the codimension of the top non--vanishing homology group of $X$ with 
coefficients in $\mathbb Z_2$ is less than or equal to two.  
We show that $G$ is solvable and a finite covering of $X$ is biholomorphic to a product
$C\times A$, where $C$ is a Cousin group and $A$ is $\{ e \}$, 
$\mathbb C$, $\mathbb C^*$, or $\mathbb C^*\times\mathbb C^*$. 

\chapter*{Acknowledgments}
\addcontentsline{toc}{chapter}{Acknowledgments}
I am indebted to my supervisor Professor Bruce Gilligan for all the thoughtful guidance he has given me 
during my PhD program.
It is his love for math that has inspired me to also be truly passionate about math.
I am thankful for all his support, understanding, and patience.

My special thanks to Dr. Fernando Szechtman for sharing his expert knowledge of Mathematics,
especially on a project in the field of Lie algebra.
I also wish to thanks Dr. Donald Stanley, Dr. Liviu Mare, and Dr. Howard Hamilton for their advice and careful reading of the manuscript.
Their suggestions have been valuable and helpful for my thesis.

Finally, the financial support of my supervisor's NSERC grant, department of Mathematics and Statistics, 
and the Faculty of Graduate Studies and Research during my PhD program allowed me to focus solely 
on my PhD program during these past years. 

\chapter*{Post Defense Acknowledgments}  
\addcontentsline{toc}{chapter}{Post Defense Acknowledgments}  

I would like to thank Professor Dr. Alan Huckleberry for carefully reading my thesis, 
making a number of constructive comments that improved it, 
and asking stimulating questions during my doctoral defense.   
I would also like to thank him for his excellent suggestions made regarding my actual presentation during my defense.   
His entire contribution is much appreciated.

\clearpage
\addcontentsline{toc}{chapter}{Dedication}
\begin{center}
{\large To my family}
\end{center}
\tableofcontents
\clearpage

\chapter{Introduction}
\pagenumbering{arabic}   
A K\"{a}hler manifold is a Hermitian manifold whose associated Hermitian form is closed.
Compact K\"{a}hler manifolds have been extensively studied from a number of view
points. For example, a compact complex surface admits a K\"{a}hler structure exactly
when its first Betti number is even; this follows from the Enriques–-Kodaira classification \cite{GH} 
and was also proved directly in \cite{Buc99}. 

\medskip
Dorfmeister and Nakajima classified K\"{a}hler manifolds
on which the group of holomorphic isometries acts transitively \cite{DN88}.
They proved that any such manifold is holomorphically isomorphic to a bounded homogeneous domain,  
a flag manifold and $\mathbb C^n$ modulo a lattice with additive structure or a product of such manifolds.       
 
\medskip
The transitivity condition of isometries is a very strong condition.
Borel-Remmert handled the compact case by showing that any connected compact homogeneous 
K\"{a}hler manifold $X$ 
which is homogeneous under the automorphism group of $X$ can be written as 
a direct product of a complex torus and a flag manifold (see Theorem \ref{torus.flag}).

\medskip
The situation in the non--compact setting, however, is not so well studied nor understood.    
In order to consider problems that are tractable throughout this dissertation
we consider a complex homogeneous manifold $X = G/H$, where $G$ is a connected
complex Lie group acting almost effectively on $X$ and $H$ is a closed complex subgroup
of $G$, such that $X$ has a K\"{a}hler structure. 
However, a classification of such homogeneous K\"{a}hler manifolds without any restriction is not realistic.

\medskip
In this dissertation we consider homogeneous K\"{a}hler manifolds $X$ with a topological invariant $d_X \leq 2$.
We mention that $d_X$ is dual to an invariant introduced by Abels \cite{Abe76}
and is defined as follows:
\[   
        d_X \; =\; \min\{\ r \ | \ H_{n-r}(X,\mathbb Z_2) \ \neq \ 0 \ \}
\]  
where by $H_{n-r}(X,\mathbb Z_2)$ we mean the singular homology of $X$ with coefficients in $\mathbb Z_2$, 
and $\dim X = n$.   
Note that since $X$ is a complex manifold, $X$ is oriented.    

\medskip
The topological invariant $d_X$ can be handled more easily in the algebraic setting, 
because algebraic groups have a finite number of connected components.
In fact, Akhiezer classified $G/H$, 
where $G$ is an algebraic group and $H$ is an algebraic subgroup 
with $d_{G/H} = 1$ in \cite{Akh77} and $d_{G/H} = 2$ in \cite{Akh83}.
On the other hand, a closed subgroup $H$ of a connected Lie group $G$ can have an infinite number of connected components,
and this phenomenon can contribute to the topology of $G/H$ and so influences $d_{G/H}$.
For example, let $G=\mathbb R$ and $H=\mathbb Z$, and $G/H = \mathbb S^1$ which is compact,
not because of any compactness of $G$, but because $H$ is infinite discrete.
This implies that the most delicate problems in working with the invariant $d_X$ occur when the isotropy is discrete.
The purpose of this dissertation is to classify K\"{a}hler $G/H$ with $d_{G/H}=2$
in the important special case when $H$ is discrete. 
In fact we prove the following theorem in chapter \ref{maintheorem}.
Note that the notion of a Cousin group is defined in Definition \ref{cousingroup}. 

\begin{theorem}  
Let $G$ be a connected complex Lie group and $\Gamma$ a discrete 
subgroup of $G$ such that $X := G/\Gamma$ is K\"{a}hler and $d_X \le 2$.  
Then the group $G$ is solvable and a finite covering of $X$ is 
biholomorphic to a product $C \times A$, where $C$ 
is a Cousin group and $A$ is $\{ e \}, \mathbb C^*$, $\mathbb C$, or $(\mathbb C^*)^2$.   
Moreover, $d_X = d_C + d_A$.    
\end{theorem}

\medskip
The dissertation is organized as follows.

In chapter \ref{bground} we give a brief review of particular known classification results
of $G/H$ some of which are used in our classification. 
In chapter \ref{tools} we introduce some tools to calculate $d_X$ mentioned above. 
We also give some other tools that are used for our classification theorem.
Note that in these chapters $H$ might not be discrete unless otherwise stated.

\medskip
Our main results are collected in chapter \ref{mainresults}.  
While classifying homogeneous K\"{a}hler manifolds $X=G/\Gamma$ with discrete isotropy and $d_X \le 2$,    
we prove that $G$ is solvable.   
This is related to a variant of a question of Akhiezer in \cite{Akh84} that is   
discussed in section \ref{Akhiezerquestion}.

\medskip   
Finally, in chapter 5 we present some projects which can be done in future.    
In the first project we introduce a method that will be used to give the full classification of 
homogeneous K\"{a}hler manifolds with top non-vanishing homology in codimension two, i.e.,
those K\"{a}hler manifolds which might not necessarily have discrete isotropy.
The next project involves complex manifolds that are homogeneous under the holomorphic 
action of real Lie groups.     
Suppose $X = G/H$ where $G$ is a real Lie group and $H$ is a closed subgroup of $G$   
such that $G/H$ has a left invariant complex structure.   
We propose to investigate globalization of the $G$ action when $d_X = 2$.  
The third project investigates extensions of the Akhiezer question in the K\"{a}hler setting.   
And the last one looks at finding some sufficient conditions in order that a holomorphically 
separable complex homogeneous manifold is Stein.

%%%%%%%%%%%%%%%%%%%%%%%%%%%%%%%%%%%%%%%%%%%

\chapter{Background}\label{bground}

In this chapter we provide background material which we will need in Chapter \ref{mainresults}   
to prove the main theorem.
In section \ref{definition} we {define} K\"{a}hler manifolds and present 
some of their properties and examples.
In section \ref{transformationgroups} we give general information about Lie groups as transformation groups. 
Stein manifolds appear in the classification theorem as fundamental building blocks.
We present some of their properties in section \ref{steinmanifolds}.
In section \ref{attemp} we introduce the classification of homogeneous manifold $G/H$ 
when there is a special restriction on the 
complex Lie group $G$.

\medskip    
One  motivation for Proposition \ref{solvmanifold} is the question of Akhiezer.  
We discuss this in section \ref{Akhiezerquestion}.\;
Let $X=G/\Gamma$ be a homogeneous K\"{a}hler manifold with discrete isotropy such that 
$\Gamma$ is not contained in any proper parabolic subgroup of $G$. 
If there are no non-constant holomorphic functions on $X$, 
then lemma \ref{radicalclosed} states that the radical orbits of $G$ are closed in $X$.
A fact which we will use to prove Theorem \ref{disc}.

%%%%%%%%%%%%%%%%%%%%%%%%%%%%%%%%%%%%%%%%%%%

\section{K\"{a}hler manifolds}\label{definition}  

In this section we define K\"{a}hler manifolds and give a brief discussion concerning their structure.   
For more information we refer the reader to \cite{HW01}, \cite{Can01} or \cite{KN63}.

A Riemannian metric on a complex manifold $M$ is a covariant tensor field $g$ of degree 2 which satisfies 
$(1)\; g(X,Y)\geq 0$, where $g(X,X)=0$ if and only if $X=0$ and $(2)\;g(X,Y)=g(Y,X)$, 
where $X$ and $Y$ are vector fields on $M$.
An almost complex structure on a manifold $M$ is a tensor field $J$ which is, at every point $m$ of $M$,
an endomorphism of the tangent space $T_m M$ such that $J^2=-1$, where $1$ 
denotes the identity transformation of $T_m M$.
The almost complex structure is complex when it is integrable.   
This means that there exist local complex coordinates on the manifold, i.e.,
an atlas of holomorphically compatible complex charts defining a complex structure on the manifold.   
Note that this is known to be equivalent to the vanishing of the Nijenhuis tensor    
\[  
       N(X,Y) \; := \; [X,Y] \; - \;  J[X,Y] \; + \; J[X,JY] \; + \; J[JX,Y]  ,    
\]  
see \cite{NN57}.   
A Hermitian metric on a complex manifold is a Riemannian metric $g$ which is invariant 
by the almost complex structure $J$, i.e., 
\[
     g(JX,JY) \; = \; g(X,Y) \;\; \mbox{for all vector fields X and Y.}
\]
Note that any paracompact real manifold has a Riemannian metric.
One can always find such metrics locally, and then use an appropriate partition of unity to extend to a global one. 
Given a Riemannian metric $g$ on any almost complex manifold note that $h(X,Y):=g(X,Y)+g(JX,JY)$ 
defines a Hermitian metric on that manifold.
Since $h$ takes its values in $\mathbb C$, we can write 
$h(X,Y) = \mathcal{S}(X,Y) + i \mathcal{A}(X,Y)$. 
Note that $S$ is symmetric and $A$ is alternating, i.e., a 2--form on $X$ called the associated K\"{a}hler form.  
If $d\mathcal{A} = 0$, then $h$ is called a {\bf K\"{a}hler metric}.      

\begin{definition}
A {\bf K\"{a}hler} manifold is a complex manifold equipped with a K\"{a}hler metric.
\end{definition}

%%%%%%%   

\subsection{Properties of K\"{a}hler manifolds}

\begin{itemize}
\item If $M$ and $N$ are K\"{a}hler then $M\times N$ is K\"{a}hler, with the product metric.
\item For any compact K\"{a}hler manifold $X$, 
the odd Betti numbers are even (see page 117, \cite{GH}).
\end{itemize}

\subsection{Examples of K\"{a}hler manifolds} \label{examplekmanifold}
\begin{itemize}
\item $\mathbb C^n$:   A Hermitian metric on $\mathbb C^n$ is given by 
$ds^2=\sum_{j=1}^n dz_j\otimes d\overline{z}_j$.    
\item  This metric induces a metric on the torus $T=\mathbb C^n/\Lambda$.
Here $\Lambda\subset\mathbb C^n$ is a complete lattice acting on $\mathbb C^n$ by translation, 
and the above metric is 
invariant under translations and thus pushes down to the quotient $T$.       
\item ${\mathbb C}{\mathbb P_n}$:
Complex projective space is defined via the following equivalence relation.
For $v, w \in\mathbb C^{n+1}\setminus\{0\}$, we let 
\[
  v\sim w \Longleftrightarrow v=\lambda w \;\;\mbox{for some } \lambda\in\mathbb C^*.
\]

Then we have the holomorphic $\mathbb C^*$--bundle
\[
  \pi: \mathbb C^{n+1}\setminus \{ 0 \} \longrightarrow ( \mathbb C^{n+1}\setminus \{0\}) / \sim \;\cong{\mathbb C}{\mathbb P_n},
\]
where ${\mathbb C}{\mathbb P_n}$ is defined to be the base of this bundle.   
Let $\sigma$ be a local section of $\pi$, i.e., $\pi\circ \sigma={\rm id}$ over an open subset 
$V$ of $\mathbb C\mathbb P_n$.  
The Fubini-Study metric on $V$ is given by 
its associated $(1,1)$ form   
\[    
       \omega \; := \; \frac{\sqrt{-1}}{2}\partial{\bar\partial} \log \|\sigma \|^2.
\]   
Let $\sigma'$ be another local section of the map $\pi$ defined over an open set $W$.  
If $V \cap W \not= \emptyset $, then $\sigma'=g.\sigma$, where $g$ is 
a nowhere vanishing holomorphic function on $V\cap W$. 
Thus we have 
\begin{eqnarray*}    
           \frac{\sqrt{-1}}{2\pi}\partial{\bar\partial}  \log\|\sigma' \|^2 
           & = & \frac{\sqrt{-1}}{2\pi}\partial{\bar\partial}(\log\|\sigma\|^2 \; + \; \log g \; + \;  \log \bar{g}) \\    
            & = & \omega+\frac{\sqrt{-1}}{2\pi}(\partial{\bar\partial} \log g\; -\; {\bar\partial} \partial \log \bar{g})  \\   
            &  =  & \omega   
\end{eqnarray*}   
Since this form is defined locally, and is independent of the section chosen, 
this construction yields a globally defined differential form of ${\mathbb C}{\mathbb P_n}$.
One can also see that this form is positive and closed (see page 31, \cite{GH}), 
so it defines a Hermitian metric on $\mathbb C\mathbb P^n$.  
Thus, the Fubini-Study metric is K\"{a}hler.

\item A complex submanifold of a K\"{a}hler manifold is K\"{a}hler by the induced K\"{a}hler structure.
As a result, any Stein manifold (embedded in $\mathbb C^n$) 
or projective algebraic variety (embedded in $\mathbb C\mathbb P_n$) 
is of K\"{a}hler type.
\item Any Riemann surface is K\"{a}hler.   
As noted above, any complex manifold $X$ admits Hermitian structures and the associated 
Hermitian form $\omega$ is closed for dimension reasons when $\dim_{\mathbb C} X =1$.  
\end{itemize} 

%-–-–-–-–-–- %%%%%%%%   

\subsection{Examples of non-K\"{a}hler manifolds}

\subsubsection{Compact Case:}
Any connected, compact, complex, homogeneous manifold which is not 
the trivial product of a flag manifold and a torus
is not a K\"{a}hler manifold (Theorem \ref{torus.flag}).
We define a properly discontinuous action of the discrete group $\mathbb Z$ on
$\mathbb C^2 \setminus \{(0 , 0)\}$ given by 
\[
(z, w)\sim 2^k (z, w)
\]
for all $(z, w)\in\mathbb C^2$ and $k\in\mathbb Z$.     
The complex surface $(\mathbb C^2 \setminus \{(0 , 0)\})/\sim$  
is compact and diffeomorphic to $\mathbb S^3\times\mathbb S^1$         
and is called the Hopf surface. 
Since the Hopf surface is not the product of a flag manifold and a torus, 
it follows from the Theorem of Borel--Remmert \cite{BR62} that the 
Hopf surface does not admit a K\"{a}hler metric.

\subsubsection{Non-compact Case:}  \label{non-compact}
Consider the three-dimensional Heisenberg group
$$   
   G \; := \;  \left\{
\left( \begin{array}{ccc}
1 & x &z  \\
0 & 1& y  \\
0&0&1
 \end{array} \right)
: x, y, z\in\mathbb C 
\right\}
$$
and consider the discrete subgroup 
$$
\Gamma:=
\left\{
\left( \begin{array}{ccc}
1 & m &k+il  \\
0 & 1& n  \\
0&0&1
 \end{array} \right)
: n, m, l, k\in\mathbb Z 
\right\}
.$$ 
The center of $G$ is the subgroup 
$$
    C \; := \; \left\{    \left( \begin{array}{ccc}
         1 & 0 &z  \\
         0 & 1& 0  \\
      0 & 0 & 1   \end{array} \right)  :  z\in\mathbb C  \right\}       
.$$
By a simple calculation one can see that the fibration $G/\Gamma\to G/C.\Gamma$ realizes $G/\Gamma$ as 
a torus bundle over $\mathbb C^*\times\mathbb C^*$.

Let $G_0$ be the real Lie subgroup
$$
 G_0:=
 \left\{
 \left( \begin{array}{ccc}
  1 & x &z  \\
  0 & 1& y  \\
  0 & 0 & 1
 \end{array} \right)
  : x, y\in\mathbb R, z\in\mathbb C 
 \right\}
$$ 
that contains $\Gamma$ cocompactly.   
Let $\mathfrak m:=\mathfrak{g}_0\cap i\mathfrak{g}_0 =\mathfrak{g}'_0\cap i\mathfrak{g}'_0$,
and so $\mathfrak{g}'_0\cap \mathfrak m\neq (0)$, where $\mathfrak{g}_0$ is the Lie algebra of $G_0$, 
and $\mathfrak{g}'_0$ is its commutator group. 
It has been proven that if a nilmanifold $G/\Gamma$ has the above form is K\"{a}hler, 
then $\mathfrak m\cap \mathfrak{g}'_0=(0)$  (see Remark 4 (b) and Example 6 (b) in  \cite{OR87} for more details).    
Since this is not so here, $G/\Gamma$ is not K\"{a}hler.

%%%%%%%%%%%%%%%%%%%    

\section{Lie groups as transformation groups on complex manifolds} \label{transformationgroups}

A complex Lie group is a set with the structure of both a complex manifold and a group 
such that the group operations are holomorphic.  
Since the sum of two solvable ideals is a solvable ideal, every complex Lie algebra contains 
a maximal solvable ideal called the radical \index{radical}.  
The corresponding connected subgroup is closed and in the case of complex Lie groups is complex.    
One sees this at the Lie algebra level, since if $\mathfrak r$ is a solvable ideal in the complex 
Lie algebra $\mathfrak g$, then a calculation shows that $\mathfrak r + i \mathfrak r$ is also a 
solvable ideal in $\mathfrak g$.  
Since the radical is maximal with these properties, one has $\mathfrak r = \mathfrak r + i \mathfrak r$ 
and thus $\mathfrak r$ is a complex ideal.   
In addition, the quotient of the complex Lie algebra $\mathfrak g$ by the complex ideal $\mathfrak r$ is 
a {\bf complex} Lie algebra say $\mathfrak s$ which is semisimple (see theorem below).
The following result was first proved by Levi for Lie algebras, but the corresponding statement for 
Lie groups is a consequence of his result.

%%%%%%%%%%%%%%%%%%%%%%%%%%%%%%% 

\begin{theorem}[Levi-Malcev theorem]\cite{Lev05}\label{Levimalcev}
Let $G$ be a simply connected, connected Lie group with radical $R$. 
Then there is a closed semi-simple subgroup $S$ in $G$
so that $G$ is the semi-direct product $G=S\ltimes R$. 
The group $S$ is unique up to conjugation.
If $G$ is complex, then so are $R$ and $S$.
\end{theorem}  

If the Lie group is not simply connected, we may have to consider the universal covering group of $G$ to apply 
the theorem. 
For example $GL_n (\mathbb C)=\mathbb C^*.SL_n (\mathbb C)$ as Lie groups, i.e., $G=R.S$, but
$R\cap S\neq (e)$.   

\subsection{Lie groups as transformation groups}  

A holomorphic left $G$--action of a complex Lie group $G$ on a complex manifold $X$ is a holomorphic map
$\phi:G\times X\to X$. 
For $g\in G$ and $x\in X$, we write $g.x$ instead of  $\phi(g,x)$ for convenience.       
Note that $g:X\to X$ is holomorphic for any fixed $g\in G$.
We say that $G$ is acting as a Lie transformation group on $X$ if $\phi$ in addition has the following two properties:
\begin{description}
\item 1) $gg'.x=g.(g'.x)$ for all $g, g'\in G$ and $x\in X$.
\item 2) $e.x=x$ for all $x\in X$. Here by $e$ we mean the identity element of $G$.
\end{description}

We call a $G$-action on X {\bf effective} if it has an additional property that whenever $g.x=x$ for all $x\in X$, then $g=e$.
By a $G$ orbit $G.x$ of a point $x$ we mean the set $G.x=\{g.x:g\in G\}$. 
For a subset $M\subseteq X$ the stabilizer $H$ of $M$ in $G$ is a {\bf subgroup} of $G$ defined by
\[
    H \; = \; {\rm Stab}_{G}(M) \; := \; \{ \ g\in G \ | \ g.M \; \subseteq \; M \ \}  .    
\]
Furthermore if $M$ is closed, then $H$ is closed, and so is a Lie subgroup of $G$.   
If $M$ is a closed complex analytic subvariety and the $G$-action is holomorphic with $G$ a complex Lie group, 
then $H$ is a {\bf complex Lie subgroup} of $G$.
Here $M$ can also be a point, in which case we denote the stabilizer group by $H_x$.
The map $gH_x\mapsto g.x$ establishes an identification between the left coset    
$G/H_x=\{gH_x: g\in G\}$ and the orbit $G.x$. 
In general if $H$ is a closed complex subgroup of the complex Lie group $G$ 
then we have the holomorphic principal bundle 
$H\to  G\to G/H$, where $G/H$ has a complex analytic quotient structure.

The ineffectivity of the holomorphic action of $G$ on $X$ 
is the normal closed complex Lie subgroup $J = \cap_{x\in X} H_x$.
This will gives us a natural action of $G/J$ on $X$ defined by $gJ.x:=g.x$.
The $G$--action is called {\bf almost effective} if the ineffectivity of the $G$--action is discrete.

{\bf Note that from now on all our group actions are almost effective, unless otherwise stated.}   
The reason for this assumption is the following.
The identity component $J^0$ of $J$ is a normal subgroup of $G$ contained in $H$.
One can replace $G$ with $\hat{G}:=G/J^0$ and $H$ with $\hat{H}=H/J^0$ in the homogeneous manifold $X:=G/H$ and 
again get the same manifold since $X=\hat{G}/\hat{H}$ but with an almost effective action of $\hat{G}$ instead.

%%%%%%%%%%%%%%%%%%%%%%%%%%%%%%% 

\subsection{Facts about discrete isotropy} 

We begin with some facts about complex homogeneous manifolds of the form 
$G/\Gamma$, where $\Gamma$ is a discrete subgroup of the complex Lie group $G$.  
These will be need later.   

Suppose $X$ is a connected complex manifold with a $G$--action of a complex Lie group $G$ on $X$    
given by the holomorphic map $\phi( ., . ): G\times X \to X$.    
For any $v\in \mathfrak g$, let $<v>_{\mathbb C}=\{tv: t\in\mathbb C\}$ be the complex span of $v$ and     
$\exp (<v>_{\mathbb C})$ be the corresponding one parameter subgroup of $G$.   
For $x\in X$ the (complex) derivative at the point $x$ defines a {\bf fundamental vector field} by       
\[   
    \xi_v (x) \; := \; \frac{d}{dt} \left|_{t=0} \phi(\exp\ tv , x)    \right.   .   
\]   
The kernel of the map $v \to \xi_v$ is the Lie algebra of the 
isotropy subgroup of the $G$--action at the point $x$.      
If $X = G/\Gamma$ with $\Gamma$ discrete, then the map $v \mapsto \xi_v$ is injective.   
So in the setting of discrete isotropy a basis of $\mathfrak g$ gives us a corresponding frame on $X$ consisting of   
$n := \dim_{\mathbb C} X$ linearly independent holomorphic vector fields at every point,  
i.e., $X$ is (complex) parallelizable.    

\begin{lemma} [\cite{Wan54}]   \label{Wang}  
Let $X$ be a connected compact  parallelizable complex manifold.   
Then $X = G/\Gamma$, where $G$ is a connected complex Lie group and  
$\Gamma$ is a discrete subgroup of $G$.
\end{lemma}    

\begin{proof}   
Let $\dim X = n$. 
Since $X$ is parallelizable,   
there are $n$ linearly independent holomorphic vector fields $X_1,\cdots,X_n$ on $X$ that
generate the tangent space $T_x X$ at any point $x\in X$.
Denote by $\mathfrak g$ the Lie algebra of all holomorphic vector fields on $X$.
For $v\in\mathfrak g$, there are holomorphic functions $f_1,\cdots,f_n \in \mathscr{O}(X)$ such that 
$v=\sum_{i=1}^n f_i X_i$.   
Every $f_i$ is constant, since $X$ is compact.       
So $\mathfrak g= \langle X_1,\cdots,X_n \rangle_{\mathbb C}$ is an $n$--dimensional complex Lie algebra.     

\medskip   
The holomorphic vector fields on any complex manifold $X$
give rise to local holomorphic 1--parameter groups.  
In general, a similar statement about global 1--parameter groups does not hold.      
However, for $X$ compact, every holomorphic vector field on $X$ can be globally integrated to a 
holomorphic 1--parameter group of $X$.           
As a consequence, the  connected, simply connected, complex Lie group $G$ 
associated to the Lie algebra $\mathfrak g$ acts holomorphically on the compact manifold $X$.  
The $G$--orbit of every point $x\in X$ is open because the vector fields generating $\mathfrak g$ 
also generate $T_x(X)$.  
Since $X$ is connected, $G$ is transitive on $X$.   
For dimension reasons the isotropy of the $G$--action on $X$ is a discrete subgroup $\Gamma$ of $G$.   
Hence $X$ is biholomorphic to $G/\Gamma$.
\end{proof}

%%%%%%%%%%%%%%%%%%%%%%%%

\section{Stein manifolds}  \label{steinmanifolds}
\begin{definition}  
A complex manifold $X$ is a Stein manifold if the following two conditions hold:  
\begin{enumerate}  
\item $X$ is holomorphically separable, i.e., for two distinct points $x$ and $y$ in $X$, 
there exists a holomorphic function $f\in\mathscr{O}(X)$ such that $f(x)\neq f(y)$.   
\item  $X$ is holomorphically convex, i.e., for any compact subset $K\subseteq X$,
the holomorphic convex hull 
$$
        \hat{K} \; := \; \{x\in X: |f(x)|\leq \sup_K |f|\;\; \mbox{ for all  } f    \in\mathscr{O}(X)\}
$$ 
is again a compact subset of $X$.
\end{enumerate}  
\end{definition}  

Here we give some well known examples:

\begin{example} 
\begin{enumerate}  
\item $\mathbb C^n$ is Stein.
\item  A domain $\Omega\subseteq \mathbb C^n$ is called a domain of holomorphy 
if there do not exist non-empty open sets $\Omega_1 \subset \Omega$ and $\Omega_2 \subset {\mathbb{C}}^n$, 
where $\Omega_2$ is connected, 
$\Omega_2 \not\subset \Omega$ and $\Omega_1 \subset \Omega \cap \Omega_2$ 
such that for every holomorphic function $f$ on $\Omega$ 
there exists a holomorphic function $g$ on $\Omega_2$ with $f = g$ on $\Omega_1$.    
A connected, open set in $\mathbb C^n$ is a domain of holomorphy if and only if it is holomorphically convex. 
Domains of holomorphy in $\mathbb C^n$ are Stein.    
\item Non--compact Riemann surfaces are Stein \cite{BS49}.  
\item Every closed complex submanifold of a Stein manifold is Stein. 
\item The Cartesian product of two Stein manifold is a Stein manifold. 
\item Suppose $X \to Y$ is a covering map with $X$ Stein.
If the fiber of this covering map is finite, then $Y$ is Stein \cite{Ste56}.  
\end{enumerate}  
\end{example}     

\bigskip 
\begin{theorem} [Embedding theorem]\cite{Rem56}
Let $X$ be a Stein manifold. 
Then there exists a proper  biholomorphic map of $X$ 
onto a closed complex submanifold of some $\mathbb C^N$.
\end{theorem}

\begin{remark}  
Any Stein manifold is also a K\"{a}hler manifold. 
One can pull back  a K\"{a}hler form on $\mathbb C^N$ by using a proper 
biholomorphic map given by the embedding theorem cited above.
\end{remark}  
 
\subsubsection{Holomorphic bundles over Stein manifolds} 
If $X\to B$ is a holomorphic covering space and the base $B$ is Stein, then $X$ is Stein \cite{Ste56}.
This is also true for holomorphic vector bundles, 
but might fail for fiber bundles with fibers $\mathbb C^n$, $n>1$, 
and with a nonlinear transition group;   
see for example Skoda \cite{Sk77} for a $\mathbb C^2$ bundle over 
an open subset in $\mathbb C$ whose total space is not Stein.          

Here we present a version of the Grauert--Oka principle which we will use in the main theorem.

\begin{theorem} [Grauert--Oka principle, Satz 6 in \cite{Gr58}] \label{GOka}  
Let $X \xrightarrow{F} B$ be a holomorphic fiber bundle, where $F$ is a complex manifold 
and $B$ is a Stein complex manifold.
Assume the group of the bundle is a (not necessarily connected) complex Lie group.    
If the bundle $X$ is continuously homotopic to another bundle $\widetilde{X}$ over the Stein base $B$, 
then it is holomorphically homotopic to $\widetilde{X}$. 
In particular, if the bundle $\widetilde{X}$ is topologically trivial, then it is holomorphically trivial.
\end{theorem}

%%%%%%%%%%%%%%%%%%   
%%%%%%%%%%%%%%%%%%   

\begin{theorem}[Satz 7 in \cite{Gr58}]   \label{Grauert}
Let $X$ be an analytic fiber bundle over a non-compact Riemann surface $Y$ 
with a connected complex Lie group as its structure group.
Then $X$ is analytically trivial.
\end{theorem}

%%%%%%%%%%%%%%%%%%%%%%%%%%%  

\begin{remark}
We later use a classical result of Serre concerning the vanishing of certain homology   
groups of Stein manifolds \cite{Ser53} (footnote, p. 59) stating that     
if $X$ is an $n$--dimensional Stein manifold,
then 
\[
      H_k (X, \mathbb Z_2) \; = \; 0
\]
for all $k > n$, where by $H_k (X, Z_2)$ we mean the $k$th homology group of $X$ with coefficients in $\mathbb Z_2$.
\end{remark}

%%%%%%%%%%%%%%%%%%%%%%%%%%%

\subsection{Lie's Flag theorem}  

A particular example of a Stein manifold as a homogeneous manifold is given by the orbits of   
a complex linear solvable Lie group 
$G$ acting holomorphically on $\mathbb P_n$.   
By Lie's Flag theorem (see \S 4.1 in \cite{Hum72}) there is a full flag 
\[ 
    \{ x \} \; \subsetneq \; L_1 \; \subsetneq \; L_2 \; \subsetneq \; \ldots \; \subsetneq \; L_n \; = \; \mathbb P_n    
\]  
that is stabilized by the $G$ action on $\mathbb P_n$.  
Suppose $X := G.y$ is some orbit of the group $G$ in $\mathbb P_n$ and set 
$k := \min \{ \ m \ | \ X \cap L_m \not= \emptyset \ \}$.    
Without loss of generality we may assume that $y \in L_k$ and thus $X \subset L_k$, since $L_k$ is 
$G$--invariant.      
Thus    
\[  
     X \; = \;  X \cap L_k \; \subset \; L_k \setminus L_{k-1} \; \simeq \; \mathbb C^m .     
\]    
As a consequence, $X$ is holomorphically separable.   
In this setting Snow proved  (Theorem 3.3, \cite{Sno85}) that $X$ is 
biholomorphic to $(\mathbb C^*)^s \times \mathbb C^t$ and thus that $X$ is Stein.

%%%%%%%%%%%%%%  

\section{Toward the classification of homogeneous manifolds}\label{attemp}

Let $G$ be a complex Lie group and $H$ a closed complex subgroup.
Borel and Remmert \cite{BR62} classified compact connected homogeneous K\"{a}hler manifolds 
(see also Theorem \ref{torus.flag}).
They showed that any such manifold is the direct product of a complex torus and a flag manifold. 
It is not realistic to classify non-compact homogeneous K\"{a}hler manifolds $G/H$ without any restriction. 

\medskip   
In this section we present helpful results when there is a condition on $G$. 
We will use these results in chapter \ref{mainresults}. 
Note that we are interested in applying the topological restriction $d_X\leq 2$ which we will introduce in section \ref{dX}.

\medskip   
In the next result and throughout this thesis we need the concept of {\it an affine cone minus its vertex}.       
For the convenience of the reader we now recall what this is.    
Suppose $Q = S/P$ is a flag manifold, where $S$ is a connected complex semisimple Lie group 
and $P$ is a parabolic subgroup.   
Then $Q$ can be $S$ equivariantly embedded into some projective space $\mathbb P_N$.   
Consider this projective space as the hyperplane at infinity in the projective space $\mathbb P_{N+1}$,   
i.e.,   
\[   
        \mathbb P_N \; = \; \{ ( 0 : z_1 : \ldots : z_{N+1} ) \in \mathbb P_{N+1} \}   
\]   
where we have used homogeneous coordinates.   
Now consider the complex manifold $X$ consisting of all the complex lines joining the point 
$(1:0:\ldots:0)$ to points in $Q$ contained in the hyperplane at infinity, but minus the point  
$(1:0:\ldots:0)$ itself and the points of $Q$.  
This is the affine cone minus its vertex over the flag manifold $Q$.   
It happens that $X = S/H$ is homogeneous under the induced $S$ action and  
$S/H \to S/P$ is a $\mathbb C^*$ bundle,   
where $H$ is the kernel of a non-trivial character $\chi: P \to \mathbb C^*$.    
As an example, if $Q = \mathbb P_N$, then $X = \mathbb C^{N+1} - \{ 0 \}$.  

\begin{theorem} \cite{AG94} \label{AG941}  
Suppose $G$ is a connected complex Lie group and $H$ is a closed complex 
subgroup such that $X := G/H$ satisfies $\mathscr{O}(X)\neq\mathbb C$ and $d_X\leq 2$. 
Let $Y :=G/ J$ be the base of the holomorphic reduction 
(see section \ref{hreduction}) of $X$ and $F := J/H$ be its fiber. 
\begin{itemize}
\item a) If $d_X =1$, then $F$ is compact and connected and $Y$ is an affine cone minus its vertex. 
\item b) If $d_X=2$, then one of the following two cases occurs: 
\subitem $b_1)$ The fiber $F$ is connected and satisfies $d_F=1$ and the base $Y$ is an affine cone minus its vertex, 
\subitem $b_2)$ The fiber $F$ is compact and connected and $d_Y = 2$; moreover, $Y$ is one of the following manifolds: 
\subsubitem $1)$ The complex line $\mathbb C$; 
\subsubitem $2)$ The affine quadric $Q_2$; 
\subsubitem $3)$ ${\mathbb P}_2 \setminus Q$ where $Q$ is a quadric curve; 
\subsubitem $4)$ A homogeneous holomorphic $\mathbb C^*$-bundle over an affine cone with its 
vertex removed. In this case $Y$ is either an algebraic principal $\mathbb C^*$ -bundle 
or is covered two-to-one by such.
\end{itemize}
\end{theorem}

%%%%%%%%%%%%%%%%%%%

\subsection{$G$ Abelian}

Let $G$ be an Abelian complex Lie group with a closed subgroup $H$. 
Thus, $H$ is normal and $G/H$ is an Abelian complex Lie group. 
In the following we define a special complex Lie group called a Cousin group, 
and then give a classification of complex Abelian Lie groups.

\subsubsection{Cousin group}

Cousin was the first to call attention to Abelian non-compact complex Lie groups,  
now called Cousin groups or toroidal groups,   
with the additional property that $\mathscr{O}(C)\cong \mathbb C$.
These groups appear in the classification of Abelian complex Lie groups \cite{Mor66}.
 
\begin{definition}[Cousin group] \label{cousingroup}
A Cousin group is a complex Lie group $C$ which admits no non-constant holomorphic function.
\end{definition}

Let $C$ be a connected complex Lie group with $\mathscr{O}(C)\cong \mathbb C$.
Since $GL_{\mathbb C}(V)$ is holomorphically separable,
there is no non-constant holomorphic homomorphism $C\to GL_{\mathbb C}(V)$ 
to the general linear group of a complex vector space $V$.
Let $V=\mathfrak C$ be the Lie algebra of $C$. 
The kernel of the adjoint representation ${\rm Ad} : C \to GL_{\mathbb C}(\mathfrak C)$ is the center of $C$
showing that {\bf a connected complex Lie group $C$ with $\mathscr{O}(C)\cong \mathbb C$ is Abelian}.    

Since a Cousin group is Abelian, the exponential map $\exp :\mathfrak C\to C$ is a
surjective homomorphism. 
The kernel $\Gamma$ of the exponential map is discrete, so $C\cong \mathfrak C/\Gamma$.
Since $\mathfrak C$ is a vector space, it is isomorphic to $\mathbb C^n$ for some $n$.
This gives us the structure of a Cousin group $C={\mathbb C}^n /\Gamma$ where $\Gamma$ is a discrete additive 
subgroup of ${\mathbb C}^n$. 
{\bf Note that we have the real torus ${\rm Span}(\Gamma)_\mathbb R/\Gamma$ inside any Cousin group.}

\subsubsection{Classification of connected, Abelian, complex Lie groups}  \label{cocaclg}

We first give a classification of connected, Abelian, real Lie group $G$ with Lie algebra $\mathfrak g$.
Since $[X,Y]=0$ for all $X, Y\in\mathfrak g$ we can consider the Lie algebra $\mathfrak g$
as a vector space over $\mathbb R$. For an Abelian connected Lie group the exponential map 
$exp: {\mathfrak g}\to G$ is a surjective homomorphism
with discrete kernel, thus a finitely generated free Abelian group (Theorem 3.6, page 25, \cite{BT85}).      
Hence, ${\rm Ker} \exp = \langle v_1,\cdots ,v_p \rangle_{\mathbb Z}$ where
$\{v_1,\cdots,v_p \}$ is a linearly independent set of vectors in $\mathfrak g$.
Let $V := \langle v_1,\cdots ,v_p \rangle_{\mathbb R}$ and choose a complementary vector subspace $W$
with ${\mathfrak g}=V\oplus W$.
Recall $\dim\mathfrak g=n$, $\dim V=p$. So $\dim W=n-p$. Then
\begin{equation*}
\begin{aligned}
 G \; \cong  \; &  \mathfrak g/{\rm Ker} \exp              \\                   
   \cong \; & (V\oplus W)/\langle v_1,\cdots ,v_p \rangle_{\mathbb R}   \\  
   \cong \; & V/\langle v_1,\cdots ,v_p \rangle_{\mathbb Z} \oplus W     \\
        = \; &      (\mathbb S ^1)^p\times \mathbb R ^{n-p}
\end{aligned}
\end{equation*}
Here $K:=(\mathbb S^1)^p$ is  the product of $p$ copies of a circle $\mathbb S^1$   
and is the unique maximal compact subgroup of $G$.    

\medskip   
For an Abelian complex Lie group $G=\mathbb C^n /\Gamma$ one has the following 
classification due to Remmert and Morimoto, e.g., see \cite{Mor66}.   
We now choose a complex vector space $W$ that is complementary to 
$V := \langle v_1, \ldots, v_p \rangle_{\mathbb C}$.  
Since $G$ is holomorphically isomorphic to $(V/\Gamma) \times W$, we can 
reduce to the setting where $\Gamma$ generates $V$ as a complex vector space   
and $W \cong \mathbb C^t$.           

\medskip   
Let $V_{\Gamma} := \langle \Gamma \rangle_{\mathbb R}$.  
Then the corresponding subgroup $K := V_{\Gamma}/\Gamma$ 
is the maximal compact subgroup of $G$.   
Set $W_{\Gamma} := V_{\Gamma} \cap i V_{\Gamma}$, the maximal complex vector subspace of $V_{\Gamma}$.   
The $W_{\Gamma}$--orbits in $G$ are isomorphic to $W_{\Gamma}/W_{\Gamma}\cap\Gamma$ and 
give a complex foliation of $G$.  
The closure of the $W_{\Gamma}$--orbit through the identity of $G$ is a subtorus $L_1$ of $K$ and we can choose a 
complementary totally real subtorus $L_2$ such that $K = L_1 \times L_2$.   
There are uniquely defined vector subspaces $\mathfrak l_i$ of $V$ such that 
$L_i = {\mathfrak l_i}/\Gamma\cap\mathfrak l_i$ and we set $U_i := {\mathfrak l_i} + i {\mathfrak l_i}$.   
Assuming $\Gamma$ generates $V$ as a complex vector space, we have   
\[  
        G \; = \; V/\Gamma \; = \; U_1/(U_1 \cap \Gamma) \times U_2/(U_2 \cap \Gamma)
\]  
Note that $ U_2/(U_2 \cap \Gamma) \cong (\mathbb C^*)^s$, since $L_2$ is totally real.   

\medskip  
We claim $ U_1/(U_1 \cap \Gamma)$ is a Cousin group, i.e., that $\mathscr{O}( U_1/(U_1 \cap \Gamma)) = \mathbb C$.  
For any $f \in \mathscr{O}( U_1/(U_1 \cap \Gamma))$ the restriction of $|f|$ to $L_1$ attains a maximum 
at some point $x\in L_1$, since the latter group is compact.  
The orbit map $\mathfrak m \mapsto G$ is holomorphic and thus the pullback of $f$ to $\mathfrak m$ 
is holomorphic and its modulus attains a maximum at any preimage of $x$.  
By the Maximum Principle this pullback function is constant.  
Thus $f$ is constant on the orbit of $f$ through $x$.   
Since this orbit is dense in $L_1$, the restriction of $f$ to $L_1$ is constant.  
Finally, by using the Identity Principle it follows from the fact that $L_1 = {\mathfrak l_1}/\Gamma\cap\mathfrak l_1$ and 
$U_1 := {\mathfrak l_1} + i {\mathfrak l_1}$ that $f$ is constant on 
 $ U_1/(U_1 \cap \Gamma)$, i.e.,  $ U_1/(U_1 \cap \Gamma)$ is a Cousin group.

%%%%%%%%%%%%%%%%%%%%%%%%%%%

\begin{example}
The set $\{(1, 0), (0, 1), (i, i\alpha)\}$ is linearly independent     
over $\mathbb R$ where $0 < \alpha < 1$ and $\alpha\notin \mathbb Q$.
Let $V := \langle (1, 0), (0, 1), (i, i \alpha) \rangle_\mathbb R$. 
One can check that
$K := V/ \langle (1, 0), (0, 1), (i, i \alpha) \rangle_\mathbb Z$ is the maximal compact subgroup
of $C := \mathbb C ^2 / \langle (1, 0), (0, 1), (i, i \alpha) \rangle_\mathbb Z$.
We claim that $\mathscr{O}(C)=\mathbb C$.
Here $\mathfrak m = \langle (1,\alpha) \rangle_\mathbb C$. The orbit of $M$ in $C$ is dense in $K$. 
Thus $L=K$ in this case and $L^\mathbb C=K^\mathbb C=C$.
Suppose $f\in \mathscr{O}(C)$. Define $\sigma:\mathbb C\to\mathfrak m, z\mapsto z(1,\alpha)$.
Consider the holomorphic map:
$$
     \mathbb C\stackrel{\sigma}{\longrightarrow}\mathfrak m\stackrel{\exp}{\longrightarrow}
     M \stackrel{f}{\longrightarrow}\mathbb C
$$
Since $M$ lies in $K$ and $f(K)$ is compact, the holomorphic function 
$f\circ\exp\circ\sigma$ is a bounded, entire function, hence constant. 
But $\sigma$ and $\exp$ are not constant.
Therefore $f|_M$ is constant and thus $f|_K$ is constant. 
However, $K$ has real co-dimension one in $C$, it follows that $f$ is constant on $C$.  
\end{example}

%%%%%%%%%%%%%%%%%%%%%%%%%%%

\subsection{G nilpotent}

We need the following definition. 
\begin{definition}    
A (principal) Cousin group tower of length one is a Cousin group.
A (principal) Cousin group tower of length $n > 1$ is a (principal) holomorphic bundle with fiber 
a Cousin group and base a (principal)
Cousin group of length $n-1$.
\end{definition}
One has the following structure theorem.   

\begin{theorem} \cite{GH78}
Let $G$ be a connected, nilpotent, complex Lie group and let $H$ be a closed complex subgroup of $G$. 
Then there exists a closed complex subgroup $J$ of $G$, containing $H$, such that
\begin{itemize} 
\item $\mathscr{O}(G/H)=\pi^*(\mathscr{O}(G/J))$, where $\pi:G/H\to G/J$ is the bundle projection
\item $G/J$ is Stein
\item the fiber $J/H$ is connected 
\item $\mathscr{O}(J/H)=\mathbb C$ 
\item $J/H$ is a principal Cousin group tower (e.g., see Theorem IV.1.5 in \cite{Oel85}).     
\end{itemize}
\end{theorem}

\begin{sketch}[Proof sketch]
One can reduce the situation to the discrete case by the following argument. 
The holomorphic reduction of $G/H$ factors through the normalizer fibration 
since the base of the normalizer fibration is holomorphically separable (Lie's flag theorem). 
So it is enough to consider the normalizer fibration.
\begin{displaymath}
    \xymatrix{
        G/H \ar[dr]^{N/H}\ar[d] &  \\
        G/J \ar[r]       & G/N }
\end{displaymath}
By assumption $G$ is a connected nilpotent Lie group, and $H^0$ is a connected subgroup of $G$. 
Matsushima \cite{Mat51} proved that for a nilpotent Lie group $G$,   
the normalizer $N:=N_G(H^0)$ of a connected subgroup is connected.   
So $G/N=\mathbb C^n$, and 
$G/H=G/N\times N/H$ by Theorem \ref{GOka}.
As a consequence, it is enough to study the fiber of the normalizer fibration
which has discrete isotropy (to be discussed in section \ref{normalizerfibration}).   
From now on we assume $H=\Gamma$ is a discrete subgroup of $G$ 
such that the smallest closed, connected, complex subgroup of $G$ containing $\Gamma$ is all of $G$.
The center $Z$ of $G$ is a positive dimensional $\Gamma$-normal subgroup of $G$, i.e., 
$Z.\Gamma$ is a closed subgroup of $G$ (see \cite{GH78}).
We have the fibration 
$$
             G/\Gamma\xrightarrow{Z.\Gamma/\Gamma} G/Z.\Gamma.
$$ 
This fibration is applied to the base of the holomorphic reduction to prove by induction that 
$G/\Gamma$ is holomorphically separable if and only if $G_{\mathbb R}/\Gamma$ is totally real
in $G/\Gamma$ if and only if $G/\Gamma$ is Stein.
For more explanation see \cite{GH78}.
\end{sketch}

%%%%%%%%%%%%%%%%%%%%%%%%%%%%%%%%%%%%%%%%%

D. Akhiezer \cite{Akh84} and K. Oeljeklaus \cite{Oel85} showed the following result that applies to the fiber of 
the holomorphic reduction of   
any complex nilmanifold $X$ and proves the remaining point in this structure result for complex nilmanifolds.    

\begin{theorem}[Theorem IV.1.5. \cite{Oel85}]
Let $X=G/\Gamma$ be a complex nilmanifold with $\mathscr{O}(G/\Gamma)\cong\mathbb C$. 
Then $X$ is a principal Cousin group tower.
\end{theorem}

\begin{sketch}[Proof sketch]  
The proof is by induction on $\dim G$ and for $\dim G = 1$, it is clear that $G/\Gamma$ is a torus. 
If $G$ is Abelian, clearly $G/\Gamma$ is a complex Lie group with $\mathscr{O}(G/\Gamma) = \mathbb C$ 
and thus $G/\Gamma$ is a Cousin group.  
So we assume the center $Z$ of $G$ satisfies $0 < \dim Z < \dim G$.  
Its orbit $Z/Z\cap\Gamma$ is closed in $G/\Gamma$ and is a connected Abelian complex Lie group.   
In the previous section we noted that $Z/Z\cap\Gamma$   
 is isomorphic to a product $\mathbb C^k \times (\mathbb C^*)^p \times C$, where $C$ is a Cousin group.  
We need to show that $\dim C > 0$ in order to apply the induction. 

\medskip 
As noted above there is a real connected subgroup $G_{\mathbb R}$ of $G$ containing $\Gamma$ cocompactly  
and since $G/\Gamma$ is not Stein, $G_{\mathbb R}/\Gamma$ is not totally real.  
Thus the ideal $\mathfrak{m} := \mathfrak{g}_{\mathbb R}\cap i\mathfrak{g}_{\mathbb R} \not= 0$. 
Let $\mathfrak{z}$ be the center of $\mathfrak{g}$, and set $\mathfrak{h}=\mathfrak{z}\cap \mathfrak{m}$.
In a nilpotent Lie algebra every non-zero ideal intersects the center non-trivially, so $\mathfrak{h}\neq 0$ 
(see page 13 in \cite{Hum72}).    
Let $H$ be the Lie group corresponding to $\mathfrak{h}$.    
Let $\mathfrak k$ denote the Lie algebra of the maximal compact subgroup $K$ of $Z/Z\cap\Gamma$.   
Note that $K$ contains a positive dimensional complex Lie subgroup, namely the orbit of $H$ 
through $e$.  
Thus $\dim C > 0$.  
Let $\widehat{C}$ denote the preimage of $C$ in $G$   and note that $\widehat{C}$ is a 
central closed complex subgroup of $G$ that has closed orbits in $G/\Gamma$.         
Then the holomorphic fibration $G/\Gamma \to G/\widehat{C}.\Gamma$ has 
the positive dimensional Cousin group $C$ as fiber and the lower dimensional 
complex nilmanifold $G/\widehat{C}.\Gamma$ as base.   
Also because of the fact that $\widehat{C}$ is central in $G$, this bundle is a principal bundle.  
Since necessarily $\mathscr{O}(G/\widehat{C}.\Gamma) = \mathbb C$, the 
result now follows by induction.   
\end{sketch}

%%%%%%%%%%%%%%%%%%%%%%%%%%%

\subsection{G solvable}    \label{CRSmanifolds}

Let $G$ be a connected solvable complex Lie group and $\Gamma$ be a discrete subgroup.
Assume there exists a closed, connected (real) subgroup $G_\mathbb R$ of $G$
containing $\Gamma$ such that $G_\mathbb R / \Gamma$ is compact. 
The triple $(G,G_\mathbb R, \Gamma )$ will be called a CRS manifold. 
The condition for the existence of such a real subgroup $G_{\mathbb R}$ in 
a given complex solvable Lie group $G$ and discrete subgroup $\Gamma$ is not known. 
However, if $G$ is a nilpotent complex Lie group, then it contains such a real Lie subgroup \cite{Mal49}. 
Note that the classification of $G/\Gamma$ for a K\"{a}hler solvmanifold is given in \cite{GO11}   
whenever $G_{\mathbb R}$ exists and $G_\mathbb R /\Gamma$ is of codimension two inside $G/\Gamma$.
C\oe ur\'{e} and Loeb present the following example \cite{CL85} which states that in the solvable case
this real subgroup might exist but $G/\Gamma$ might not even be K\"{a}hler.

\begin{example}\label{CL}
Let $G_K = K\ltimes K^2$, where $K=\mathbb Z$ or $\mathbb C$ with the group operation given by 
\[
          (z,b)\circ(z',b') \; := \; ( z+z' , e^{Az'}b + b' )   
\]
where $z, z'\in K$, $b,b'\in K^2$ and $A$ is the real logarithm of the matrix 
 $    \left( \begin{array}{cc}  2 & 1 \\   1 & 1 \\  \end{array} \right)  $.  
Then $G_\mathbb C$ is a connected, simply connected three dimensional complex Lie group 
that contains the discrete subgroup
$G_\mathbb Z$ such that its holomorphic reduction is given by 
$$
         {G_\mathbb C}/{G_\mathbb Z} \to {G_\mathbb C}/{G'}_{\mathbb C}.{G_\mathbb Z}
$$
with the base biholomorphic to $\mathbb C^*$ and the fiber to $\mathbb C^*\times\mathbb C^*$. 
\end{example}

Huckleberry and E. Oeljeklaus showed that the base of the holomorphic reduction is Stein (Theorem, p. 58 in \cite{HO86}).
The following two theorems give valuable information for holomorphic reductions of K\"{a}hler solvmanifolds.

%%%%%%%%%%%%%%%%%%%%%%%%%%%

\begin{theorem}\cite{OR88}\label{OeRi88}
Let $X$ be a K\"{a}hler solvmanifold and let 
\[   
          X \; \stackrel{F}{\rightarrow} \;  Y 
\]    
denote the holomorphic reduction of $X$. 
Then $Y$ is a Stein manifold and F is a Cousin group.   
Moreover the first fundamental
group $\pi_1 (X)$ contains a nilpotent subgroup of finite index.
\end{theorem}

%%%%%%%%%%%%%%%%%%%%%%%%%%%

\begin{theorem}\cite{GO08}  \label{GO08}
Suppose $G$ is a connected, solvable, complex Lie group and $H$
is a closed complex subgroup of $G$ with $X := G/H$ a K\"{a}hler manifold. 
Let 
\[
     G/H\; \longrightarrow \;   G/I
\]
be the holomorphic reduction. Then there is a subgroup of finite index $\hat{I}\subset I$
such that the bundle
\[
\hat{X}:=G/\hat{I}\cap H \to G/\hat{I}
\]
is a holomorphic $I^{0}/H\cap I^{0}$- principal bundle 
and represents the holomorphic reduction of $\hat{X}$. 
\end{theorem}

\subsection{$G$ semi-simple}\label{semisimple}
    
Let $G$ be a semisimple complex Lie group and recall that $G$ admits the structure of  
a linear algebraic group.   
A Borel subgroup of $G$ is a maximal connected solvable subgroup of $G$.
A parabolic subgroup is a subgroup of $G$ containing a Borel subgroup.     
If $P$ is a parabolic subgroup of $G$, then the manifold $G/P$ is a flag manifold.
For information on flag manifolds we refer the reader to \cite{Wol69} and \S 3.1 in \cite{Akh95}.   

\begin{theorem}[\cite{BO88} and Corollary 4.12 in \cite{GMO13}]   \label{BO}
Let $G$ be a semi-simple complex Lie group and $H$ a closed complex subgroup of $G$. 
Then the following conditions are equivalent: 
\begin{description}
\item (i) H is algebraic.
\item (ii) the homogeneous space $G/H$ is K\"{a}hler.
\end{description}
\end{theorem}

One important consequence of this theorem is that when $X$ is
a K\"{a}hler manifold under the transitive action of a semisimple group with discrete isotropy, 
then the isotropy group is automatically a finite set of points.
One should note that since there is no known classification of discrete subgroups of complex Lie groups, 
even in low dimensional cases such as $SL(2,\mathbb C)$, 
the classification of semisimple homogeneous manifolds is impossible at the moment.

%%%%%%%%%%%%%%%%%%%%%%%%%%%
%%%%%%%%%%%%%%%%%%%%%%%%%%%

\subsection{G linear algebraic}

\begin{theorem}  \label{linearalgebraic}   
Let $G$ be a connected complex linear algebraic group and $H$ an algebraic subgroup of $G$. 
Set $X:=G/H$.
\begin{enumerate}
\item \cite{Akh83} 
Suppose  $d_X=2$. 
Then the space $X$ is a twisted product $X=G\times_{P}F$.
Here $P$ is a parabolic subgroup of $G$ and 
the manifold $F$ is isomorphic to $\mathbb A^1$, $\mathbb C^*\times \mathbb C^*$, 
or $\mathbb{P}_{2}\setminus Q$ with $Q$ a quadric,
and the transitive operation $P\times F\to F$ is given by affine transformation, 
by group translations, or by projective transformations preserving $Q$.  
\item \cite{Akh77} If $d_{G/H}=1$, 
then there exists a parabolic subgroup $P$ of $G$ containing $H$ such that $G/H\to G/P$ is 
a $\mathbb C^*$-bundle over the flag manifold $G/P$.
\end{enumerate}
\end{theorem}

%%%%%%%%%%%%%%%%%%%%%%%%%%% 

\subsection{$G$ reductive}

A complex Lie group is reductive if it is the complexification of a totally real maximal compact subgroup.   
Then a finite covering of any reductive complex Lie group has the form $S \times R$, where  $R=(\mathbb C^*)^k$.   
In passing, we note that here $S\unlhd G$. 
The following result had been proved earlier.  

\begin{theorem}[\cite{Mat60}, \cite{Oni60}]   
Let $G$ be a reductive complex Lie group, H be a closed complex subgroup.
Then $G/H$ is Stein if and only if $H$ is reductive.
\end{theorem}

In passing we note that the K\"{a}hler setting is characterized by the following.  

\begin{theorem} [Theorem 5.1 in \cite{GMO13}]  
Suppose $G$ is a reductive complex Lie group with Levi--Malcev decomposition $G=S.R$ and 
let $H$ be a closed complex subgroup of $G$.   
Then $G/H$ is K\"{a}hler if and only if $S.H$ is closed in $G$ and $S\cap H$ is an algebraic subgroup of $S$.  
\end{theorem}

\begin{theorem} (Lemma in \S 2.4 in \cite{BO73}) \label{reduct}
Assume $G$ is a reductive complex Lie group.
Let $H$ be a closed subgroup with $\overline{H}$ as its Zariski closure.
Every $H$-invariant holomorphic function on $G$ is necessarily $\overline{H}$-invariant. 
\end{theorem}   

In other words, for $G$ reductive the holomorphic reduction $G/J$ of the complex homogeneous space 
$G/H$ is also the holomorphic reduction for $G/\overline{H}$ and so the holomorphic reduction map factors  
through $G/\overline{H}$ as is given in the following commutative diagram   
\[  
     \begin{array}{ccccc}    G/H & & \longrightarrow & & G/\overline{H} \\  
               & \searrow & & \swarrow & \\  
               & & G/J & & \end{array}  
\]  
                 
In fact $J$ is algebraic (see Satz in \S 2.5 in \cite{BO73}).

%%%%%%%%%%%%%%%%%%%%%%%%%%%

\section{Akhiezer's question in K\"{a}hler setting}\label{Akhiezerquestion}

A K\"{a}hler manifold $X=G/\Gamma$ with discrete isotropy $\Gamma$
and $d_X\leq 2$ is a solvmanifold (see proposition \ref{solvmanifold}). 
This interesting phenomenon which we prove is related to      
a question of Akhiezer \cite{Akh84}.
The original question concerned the existence of analytic hypersurfaces in complex homogeneous manifolds.   
We consider the following variant of this question in the K\"{a}hler setting.

%%%%%%%%%%%%%%%%%%%%%%%%%%%%%%%%

\subsubsection{Modified Question of Akhiezer.} 

Suppose $X := G/H$ is a K\"{a}hler homogeneous manifold that satisfies 
\begin{description}
\item (i) $\mathscr{O}(X)=\mathbb C$
\item (ii) There is no proper parabolic subgroup of $G$ which contains $H$
 \end{description}  
Is $G$ then solvable?   
  
Note that if $G$ is known to be solvable, then it follows that $X$ is a Cousin group by Theorem \ref{OeRi88},    
since in that theorem the base of the holomorphic reduction of $X$ will be a point.    
Hence $X$ will be isomorphic to the fiber of its holomorphic reduction which is a Cousin group. 

\medskip   
We will need the following.   

\begin{theorem}(Corollary 6, p. 49 in \cite{HO84}]  \label{HO84}
Let $G/H$ be a complex homogeneous manifold, and $N:=N_G(H^0)$. 
If $\mathscr{O}(G/H) = \mathbb C$ and if $N\neq G$, then $H$ is contained in a proper parabolic 
subgroup $P$ in $G$.
\end{theorem}

Next we prove that the radical orbits in some special setting are closed.
We use this fact to prove proposition \ref{solvmanifold}.
Note that the following example shows the radical orbits might not be closed in general.   

\begin{example}
Consider $G := \mathbb S^1 \times \mathbb R$, where we consider the 
circle as complex numbers of modulus one.  
Let $\Gamma$ be the discrete subgroup $\{ (\pi i k \alpha, k ) | k\in\mathbb Z \}$, where $\alpha$ is an irrational number.  
Note that $\Gamma . \mathbb R$ is not closed in $G$.    
For a complex version of this embed the $\mathbb S^1$ into the diagonal matrices in 
$S = SL(2,\mathbb C)$ which is biholomorphic to $\mathbb C^*$.
Set $G = S \times \mathbb C$. One can see that $\Gamma.R$ is not closed in $G$ where $R$ denotes the radical of $G$. 
\end{example}

\begin{lemma}  \label{radicalclosed}    
Suppose $G$ is a mixed connected complex Lie group, i.e., $G = S.R$ is a Levi--Malcev decomposition, 
where both $S$ and $R$ have positive dimension.  
Let $X:=G/\Gamma$, where the discrete isotropy $\Gamma$ is not contained in any proper parabolic subgroup of $G$. 
If there is no non-constant holomorphic function on $X$, then the radical orbits of $G$ are closed in $X$.
\end{lemma}

\begin{proof}
All $R$ orbits are closed if and only if $R.\Gamma$ is closed. 
To prove that $R.\Gamma$ is closed we assume it is not and derive a contradiction.
Let $H:={\overline{R.\Gamma}}$ where by bar we mean the topological closure.   
Auslander, using earlier results of Zassenhaus \cite{Aus63} 
proved that $H^0$ is solvable.    
If $H^0$ is not a complex Lie group, let $H_1$ be the connected Lie subgroup of $G$ 
associated to the complex Lie algebra $\mathfrak h_1:=\mathfrak h +i\mathfrak h$,   
where $\mathfrak h$ denotes the Lie algebra of $H^0$.  
Note that $H_1$ is solvable and complex.   
If $H_1$ is not closed, we continue this process which has to stop after a 
finite number of steps, because $G$ has finite dimension.

\medskip    
We let $I$ be the {\bf 'minimal' closed connected complex solvable} Lie group obtained in this manner. 
Define $J := N_G (I)$.
Note that $I$ is not the radical, since, by assumption, $R$ does not have closed orbits and so $J \not= G$.  
Now consider the fibration $G/\Gamma \to G/J$.  
Since $J$ is, by definition, the normalizer of a connected group 
we see that $G/J$ is an orbit in some projective space (see section \ref{normalizerfibration}).
Since $\mathscr{O}(G/\Gamma)\cong \mathbb C$, it then follows that $\mathscr{O}(G/J)\cong \mathbb C$.
Theorem \ref{HO84} applies, so $J$ is contained in a proper parabolic subgroup of $G$.
If we prove that $\Gamma\subset J$ then $\Gamma$ 
is contained in this proper parabolic subgroup which is a contradiction to our assumptions.
To see this first note that for any $g\in H$ we have $gH^0 g^{-1} = H^0$ and so we have $H^0 \subseteq gIg^{-1}\cap I$.
But $I$ is the minimal closed connected complex solvable Lie group that contains $H^0$,
so $gIg^{-1} = I$, hence $H \subseteq J$.   
But $\Gamma \subset H$. %As a conclusion $N_G(I)=G$.    
Thus $\Gamma \subset J$ and this contradiction proves that  $R.\Gamma$ is closed in $G$.
\end{proof} 

\medskip  
We first look at the case where $G = S \times R$, i.e., $G$ is the trivial product group of a maximal 
semisimple subgroup $S$ with the radical $R$ of $G$.  

\begin{lemma} [Lemma 6 in \cite{AG94}]   \label{lem6}  
Suppose $G$ is a complex Lie group whose Levi decomposition is a direct product $S\times R$. 
Let $\Gamma$ be a discrete subgroup of $G$ such that $S\cap \Gamma$ is finite. 
Then $\Gamma$ is contained in a subgroup of the form $A\times R$, 
where $A$ is an algebraic subgroup of $S$ such that its identity component $A^0$
is solvable.
\end{lemma}  

\medskip 
  
We will use the following theorem in chapter \ref{mainresults}.

\begin{theorem}[p. 116 in \cite{OR87} \& Theorem 2 in \cite{Gi04}] \label{Gi041}
Suppose $G$ is a connected simply connected complex Lie group
whose Levi-Malcev decomposition is a direct product $G = S \times R$.     
Suppose $\Gamma$ is a discrete subgroup of $G$ such that $X:=G/\Gamma$ satisfies the conditions 
\begin{itemize}
\item [] (a) $X$ is K\"{a}hler
\item [] (b) $\Gamma$ is not contained in a proper parabolic subgroup of $G$ and 
\item [] (c) $X$ does not possess non-constant holomorphic functions
\end{itemize}
Then $S = \{e\}$.   
That is, $G/\Gamma$ is a Cousin group.
\end{theorem}   

\medskip
Depending on the analytic behavior of the radical orbits there is an answer to the Akhiezer's question.
If the radical orbits do not possess any non-constant holomorphic functions then we have the following theorem
which basically presents us to the situation of the Theorem \ref{Gi041}. 

\begin{theorem}[Theorem 3 in \cite{Gi04}]  \label{Gi04}
Suppose $G$ is a connected complex Lie group and $\Gamma$ is a discrete
subgroup of $G$ such that $G/\Gamma$ satisfies the condition (a), (b), and (c)  of the Theorem \ref{Gi041}.
Let $R$ denote the radical of $G$. 
Assume the typical radical orbit has no non-constant holomorphic function, 
that is, $\mathscr{O}(R.\Gamma/\Gamma) = \mathbb C$.   
Then $S = \{e\}$, that is, G is solvable. 
\end{theorem}

If the $R$ orbits possess non-constant holomorphic functions,    
then the answer to Akhiezer's question in the general K\"{a}hler setting is not known.    
However, enough is known in certain cases that we are able to use these results   
to prove the classification in our main theorem.    
The remaining special case that we need is when $G = S \ltimes R$ is a semidirect 
product and the radical $R$ has complex dimension two.  
Then one has the following theorem which is extracted from the proof of Lemma 8 in \cite{AG94}.

\begin{lemma}  \label{lem8}   
Let $G$ be a complex Lie group with Levi-Malcev decomposition $G=S\ltimes R$ with 
$\dim_{\mathbb C} R=2$. 
Let $\Gamma$ be a discrete subgroup of $G$ such that $X=G/\Gamma$ is K\"{a}hler.    
Then $\Gamma$ is contained in a proper subgroup of $G$    
of the form $A \ltimes R$, where $A$ is a proper algebraic subgroup of $S$.   
\end{lemma}   
  
\begin{proof}  
Without loss of generality we may assume that $G$ is simply connected.  
Suppose that the adjoint action of $S$ on $R$ is trivial.  
Then $G=S \times R$ and $S \cap \Gamma$ is finite by Theorem \ref{BO}.   
The result follows by Lemma \ref{lem6}, i.e., $\Gamma \subseteq A \ltimes R$, 
where $A$ is a proper algebraic subgroup of $S$.

\medskip 
Assume now that the adjoint action of $S$ on $R$ is not trivial.   
Then we claim that $R$ is Abelian.   
If not, then $R'$ is one--dimensional and $S$--invariant.   
Hence the $S$--action on $R'$ is trivial and since $S$ is semisimple, 
this action is completely reducible.   
So there is a one dimensional subspace of $R$ complementary to $R'$ 
that is also $S$--invariant.   
But this action is also trivial which gives a contradiction to the $S$--action on $R$ 
being non--trivial.   
So $R$ is Abelian and the $S$--action on $R$ is irreducible.  

\medskip       
We next note that $G$ has the structure of a linear algebraic group.  
We see this in the following way.   
Consider a semisimple complex Lie group $S$ with an irreducible two dimensional representation 
$R: S \to GL(2,\mathbb C)$.  
Since $S$ is perfect, its image is too and thus this image must be isomorphic to $SL(2,\mathbb C)$ which is 3 dimensional.  
By the rank--nullity theorem the kernel of $R$ has codimension 3 and is a normal subgroup of $S$.   
Thus $S$ decomposes as a  
direct product $S = S_1 \cdot S_2$, where $S_2$ acts trivial and $S_1 \simeq SL(2,\mathbb C)$ 
with the usual action on $\mathbb C^2$.   
Any element $s\in S$ has the form $s_1 \cdot s_2$ with $s_1$ having different eigenvalues 
and $s_2$ acting trivially on $R$. 
So $G = S_2\times (S_1 \ltimes \mathbb C^2)$.   
Since $S_2$ and $S_1 \ltimes \mathbb C^2$ are both linear algebraic, it follows that $G$ has a linear algebraic structure.   

\medskip  
We claim that $\Gamma$ is contained in a proper algebraic subgroup of $G$.    
We assume not, i.e., that $\Gamma$ is Zariski dense in $G$ and derive a contradiction.   
Then $\pi(\Gamma)$ is Zariski dense in $S$, where $\pi:G \to S$ denotes the projection. 
It follows that $\pi(\Gamma)$ contains a free group with two generators that is also Zariski dense in $S$ \cite{Tit72}.  
We pick $\pi$--preimages of two of these generators in $\Gamma$ and let $\widetilde{\Gamma}$ be the 
subgroup of $\Gamma$ generated by these elements.  
There is a torsion free subgroup $\widetilde{\Gamma}_1$ of finite index in 
$\widetilde{\Gamma}$ (Corollary 6.13, page 95, \cite{Rag08}).  
Then $\pi(\widetilde{\Gamma}_1)$ is also Zariski dense in $S$.  
For any $g\in G$ the intersection $\widetilde{\Gamma}_1 \cap gSg^{-1}$ is finite, since the $gSg^{-1}$--orbit 
in the K\"{a}hler manifold $G/\widetilde{\Gamma}_1$ has finite isotropy \cite{BO88}.   
Since $\widetilde{\Gamma}_1$ is torsion free, we have $\widetilde{\Gamma}_1 \cap gSg^{-1} = \{ e \} $.   
Thus 
\[  
        \widetilde{\Gamma}_1 \; \setminus \; \{ e \} \; \subset \; G \; \setminus \; \bigcup_{g\in G} gSg^{-1} 
        \; \subset \; G \; \setminus \; \pi^{-1}(S \; \setminus \; Z) \; = \; \pi^{-1}(Z) , 
\]   
where $Z$ is a proper Zariski closed subset of $S$ given by Lemma 7 in \cite{AG94}.  
Then $\pi(\widetilde{\Gamma}_1) \subset Z$, contradicting the fact that $\pi(\widetilde{\Gamma}_1)$ is 
Zariski dense in $S$. 
Hence $\Gamma$ is not Zariski dense in $G$, i.e.,
$A:= \overline{\pi(\Gamma)}$ is a proper algebraic subgroup of $S$, so $\Gamma \subset A\ltimes R$ as desired.
\end{proof}

This has the following consequence which we use later.  

\begin{theorem}   \label{AG942} 
Suppose $G$ is a connected complex Lie group with Levi-Malcev decomposition $G=S\ltimes R$ with 
$\dim_{\mathbb C} R=2$. 
Let $\Gamma$ be a discrete subgroup of $G$ such that $X=G/\Gamma$ is K\"{a}hler,   
$\Gamma$ is not contained in a proper parabolic subgroup of $G$ and $\mathscr{O}(G/\Gamma) 
\simeq \mathbb C$.       
Then  $S=\{e\}$, i.e., $G$ is solvable.   
\end{theorem}  

\begin{proof}  
The radical orbits are closed by lemma \ref{radicalclosed}.   
By lemma \ref{lem8} the subgroup $\Gamma$ is contained in a proper subgroup of $G$ of the form 
$A \ltimes R$, where $A$ is an algebraic subgroup of $S$.  
Thus there are fibrations    
\[  
         G/\Gamma \; \longrightarrow \;  G/R.\Gamma  \; \longrightarrow \; S/A ,  
\]  
where the base $G/R.\Gamma = S/\Lambda$ with $\Lambda := S \cap R.\Gamma$.  
If $A$ is reductive, then $S/A$ is Stein and we get non-constant holomorphic functions 
on $X$ as pullbacks using the above fibrations.    
But this contradicts the assumption that $\mathscr{O}(X) \simeq \mathbb C$.  
If $A$ is not contained in a proper reductive subgroup, then Theorem in \S 30.4 in \cite{Hum75} applies and     
$A$ is contained in a proper parabolic subgroup of $S$    
and so is $\Gamma$, thus contradicting the assumption that this is not the case.  
\end{proof}  

We finish this chapter with a very well known theorem of Borel and Remmert 
concerning compact complex K\"{a}hler manifold.
We will refer to this theorem through this thesis.

\begin{theorem}\cite{BR62}\label{torus.flag}
Let $X$ be a connected compact complex K\"{a}hler manifold such that 
the group of all holomorphic transformations of $X$ is transitive.
Then $X$ is the product of a complex torus and a flag manifold.
\end{theorem}

%%%%%%%%%%%%%%%%%%   

\chapter{Basic Tools}\label{tools}

\section*{Fibration Methods}

One of our standard tools throughout is the use of fibrations, in particular, 
those that arise as homogeneous fibrations, i.e., 
\[   
      G/H \; \xrightarrow{J/H} \; G/J .      
\] 
where $G$ is a complex Lie group, and $J$ and $H$ are closed complex
subgroups with $H$ contained in $J$.  
Let $H^0$ be the identity component of $H$, and $\widetilde{H}$ be the largest subgroup of $H$ normal in $J$. 
The group of the bundle is $J/\widetilde{H}$ acting in $J/H$ as left translations (see the Theorem on page 30, \cite{Ste57}).    
If this action is trivial the bundle splits as the trivial product $G/J\times J/H$.
 
Here we give a brief explanation of some fibrations that we will use in later chapters. 
For further discussion we refer the reader to the references \cite{HO84}, \cite{Hu90} and \cite{HO81}.

%%%%%%%%%%%%%%%%%%%%%%%%

\section{Normalizer fibration}  \label{normalizerfibration}

The normalizer of a subgroup of a group has been a useful tool in group theory for a long time.  
Tits was the first who analyzed the structure of compact complex homogeneous manifolds 
using the normalizer subgroup \cite{Tit62}.  
Given a complex homogeneous manifold $X=G/H$, where $G$ is a (connected) complex Lie 
group and $H$ is a closed complex subgroup, let $N := N_G(H^{0})$ be the normalizer in $G$ 
of the identity component $H^{0}$ of $H$. 
Then $H\subset N$ and $N$ is a closed complex subgroup of $G$.    
One then has the fibration $G/H \to G/N$.  
This fibration is useful because its base $G/N$ is an orbit in some projective space.
To see this consider the adjoint representation ${\rm Ad} : G\to GL(\mathfrak{g})$, 
where $\mathfrak{g}$ is the Lie algebra of $G$. 
Let $\mathfrak{h}$ be the Lie algebra of $H$.
Since $\mathfrak{h}$ is a vector subspace of dimension $k$ in the vector space 
$\mathfrak{g}$ of dimension $n$, we may consider $\mathfrak{h}$ as a point in Grassmann manifold ${\rm Gr}(k,n)$, 
and $N$ as its isotropy group.
As a conclusion we can always consider the base of a normalizer fibration $G/N$ as an orbit in some Grassmann manifold.
By Pl\"{u}cker embedding theorem this Grassmann manifold can be equivariantly embedded 
in a complex projective space $\mathbb P_m$.

Since $H^0$ is normal in $N$, it acts trivially on the fiber $F = N/H$, 
and we can write the fiber as $F=\hat{N}/\Gamma$,
where $\Gamma :=H/H^0$ is a discrete subgroup of 
the complex Lie group $\hat{N} :=N/H^0$.
Thus, the fiber of the normalizer fibration is always complex parallelizable.   
If the fiber of the normalizer fibration
is also compact and K\"{a}hler, then the lemma below shows that it is a torus.

\begin{lemma} [Corollary 2 in \cite{Wan54}]  \label{torus}
Let $X$ be a compact complex parallelizable K\"{a}hler manifold. 
Then $X$ is a torus, i.e. $X = \mathbb C^n/\Gamma$, where $\Gamma$ is a co-compact, discrete additive subgroup.
\end{lemma}

\begin{proof}  
By Lemma \ref{Wang} we have $X = G/\Gamma$. 
To prove the lemma it is enough to prove that the Lie algebra $\mathfrak g$ is Abelian and  
so $G \cong \mathbb C^n$.
Let $\omega_1, \dots, \omega_n$ be a basis of right-invariant holomorphic
1-forms on $G$.   
Because these forms are right G-invariant they are right $\Gamma$-invariant so we can consider them as forms on $X$.  
Since $X$ is compact and K\"{a}hler, holomorphic forms on $X$ are closed.   
Hence for the basis $X_1,\cdots,X_n$ of invariant vector fields on $X$ one has     
\[   
     0 \; = \; d\omega_i(X_j,X_k) \; = \; \frac{1}{2} \{ X_j (\omega_i(X_k)) - 
     X_k (\omega_i (X_j)) \}-\omega_i([X_j,X_k])= -\omega_i([X_j,X_k])
\]   
for all $i, j$, and $k$ (see page 36, \cite{KN63}). 
This implies  $[X_j,X_k]=0$ for all $j$ and $k$, and therefore the Lie algebra $\mathfrak g$ is Abelian.
\end{proof}

\medskip
The base of the normalizer fibration of a compact homogeneous manifold is a flag manifold. 
If a homogeneous complex manifold is not compact, it might happen that the base is compact.   
The next lemma is the Borel Fixed Point Theorem and is stated here for future reference.

\begin{lemma}\label{flag} %\cite{Hu90}
If the base $G/N$ of the normalizer fibration is compact, 
then it is a flag manifold, i.e., homogeneous under the effective action of 
a semisimple group. 
\end{lemma}

%%%%%Commutator fibration%%%%%%%%%%%

\section{Commutator fibration}   \label{commutatorfibration}

The base of the normalizer fibration is an orbit in some projective space. 
In this setting $G$ is represented as a complex linear group with $\mathfrak g$ as its Lie algebra.   
Let $\overline{G}$ be the algebraic closure of $G$ and note that the  
corresponding $\overline{G}$--orbit $\overline{G}/\overline{H}$ contains   
the $G$--orbit $G/H$.   
By abuse of language we are using $\overline{H}$ to denote the isotropy of   
the $\overline{G}$ orbit and not the algebraic closure of $H$.  
Since $\overline{G}$ is an algebraic group acting algebraically on the projective space,   
its commutator subgroup $\overline{G}'$ is also algebraic and has closed orbits.    
This gives rise to the following diagram  
\[   
     \begin{array}{ccc}    G/H & \longrightarrow & \overline{G}/\overline{H}  \\  
     \downarrow & & \downarrow \\  
     G/I & \longrightarrow & \overline{G}/\overline{G}' . \overline{H}  \end{array}   
\]    
Let $G'$ be the commutator group of $G$ with $\mathfrak g '$ as its Lie algebra.  
By Chevalley's theorem (Theorem 13, page 173, \cite{Che51}), 
one has $\mathfrak{g}' = \overline{\mathfrak{g}}'$ and 
thus $G' = \overline{G}'$.  % is acting as an {\bf algebraic group} on a projective space.  
Now we claim that $I = G'.H = G\cap\overline{G'}.\overline{H}$. 
Let $g=g_1 . g_2 \in G\cap \overline{G'}.\overline{H}$, where $g_1 \in \overline{G'}= G'$ and $g_2 \in \overline{H}$. 
But since $g_2\in G\cap \overline{H} = H$ we conclude $g = g_1 . g_2 \in G' . H$ and 
hence $G\cap \overline{G'}.\overline{H} \subseteq G'.H$.
The opposite inclusion is also true and we get the equality as desired. 
As a consequence, $G/I = G/G'.H$ is an orbit in the affine algebraic Abelian group $\overline{G}/\overline{G}' . \overline{H}$.  
Since the latter group is Stein, we see that $G/G'.H$ is also Stein and we have the commutator group fibration 
\[  
      G/H \; \longrightarrow \; G/G'. H  
\]  
If $G$ is also solvable, then the commutator group is unipotent. 
So the fiber of the above fibration is algebraically isomorphic to some affine space $\mathbb C^n$.

%%%%%%%holomorphic reduction%%%%%%%%%%%

\section{Holomorphic reduction} \label{hreduction}

Let $X:=G/H$ be a homogeneous complex manifold, where $G$ is a (connected) complex Lie group,
and $H$ a closed complex subgroup.
Define an equivalence relation $\sim$ on $X$ by 
$$x\sim y\Longleftrightarrow f(x) = f(y)\;\;\;\forall f\in\mathscr{O}(X).$$
We have the natural map 
\[
\pi:X\to X/\sim=G/J
\]
where 
\[  
    J \; := \; \{ \ g \in G \ | \ f(gH) \ = \ f(eH) \quad\forall\quad f \in \mathscr{O}(G/H) \ \}  .  
\]   
Then $J$ is a closed complex subgroup of $G$ and contains $H$.  
The map $\pi$ is called the holomorphic reduction of $X$.   
Note that for an arbitrary complex manifolds (non-homogeneous), $X/\sim$ might not be locally compact and so 
might not have the structure of a complex space.
Also, the base of a holomorphic reduction need not to be Stein and 
one may not have $\mathscr{O}(J/H)\neq \mathbb C$.
\begin{example}
In $G:=SL(2,\mathbb C)$ we let 
\[  
       H \; := \;  \left\{  \left( \begin{array}{cc}  1 & k \\  0 & 1 \\ \end{array}   \right)  
         :\;\;  k\in\mathbb Z  \right\}
\]  

The algebraic closure of $H$ is  
\[  
         \overline{H} \; = \; \left\{  \left(  \begin{array}{cc} 1 & w \\ 0 & 1 \\  \end{array}   \right)  
            :\;\;  w\in \mathbb C   \right\}
\]   
We have the intermediate fibration
\[
       G/H \to G/\overline{H} \to G/J
\]
where $J$ is the subgroup of $G$ that appears in the definition of the holomorphic reduction of $G/\overline{H}$.
Since $G$ is reductive, the theorem of Barth and Otte (Theorem \ref{reduct}) applies and
every $H$-invariant holomorphic function on $G$ is necessarily $\overline{H}$-invariant.   
Note that $\overline{H}\subseteq J$.
Hence the holomorphic reduction of $G/\overline{H}$
and $G/H$ are the same. In this case $G/\overline{H} \cong \mathbb C^2 - \{(0,0)\}$ which is 
holomorphically separable, i.e., $J=\overline{H}$.

So we see that the base of the holomorphic reduction is not necessarily Stein. 
The fiber $J/H=\mathbb C^*$ is Stein and so not necessarily a Cousin group.
\end{example}

The holomorphic reduction $\pi$ has the following properties:
\begin{itemize}
\item The homogeneous complex manifold $G/J$ is holomorphically separable  
\item $\pi^{*} \mathscr{O}(G/J) \simeq \mathscr{O}(G/H)$, i.e.,    
all holomorphic functions on $G/H$ arise as the pullbacks  of holomorphic functions on $G/J$   
via the holomorphic map $\pi$. 
\end{itemize}

%%%%%%%%%%%%%%%%%%   

\section{Topological invariant $d_X$}  \label{dX}

By $d_X$ we mean the co-dimension of the top (singular) homology group of $X$
with coefficients in $\mathbb Z_2$ that does not vanish.

\begin{definition} 
Let $X$ be an oriented manifold with $\dim_\mathbb R X = n$.    
Define
\[
     d_X \; := \; \min \{ \ k \ : \ H_{n-k}(X,\mathbb Z_2) \ \neq \ 0 \ \} .        
\]
\end{definition}

Note that $d_X = 0$, exactly if the manifold $X$ is compact. 
Otherwise, $d_X$ is positive.
The integer $d_X$ is dual to an invariant introduced by H. Abels \cite{Abe76}.  
The following theorem is a good tool to calculate $d_X$, when $X$ has {\bf connected isotropy}.   
The reader should note that this does not apply in the setting of discrete isotropy and one needs   
the fibration lemma given in the next section in order to handle that setting.           

\begin{theorem}[Covariant fibration \cite{Mos55}]    \label{MOSTOW} 
Let $G$ be a connected real Lie group and $H$ a connected closed Lie subgroup.  
Further, suppose $K$ is a maximal compact subgroup of $G$ and $L$ is 
a maximal compact subgroup of $H$ contained in $K$.  
Then $G/H$ can be retracted by a strong retraction to the compact
space $K/L$, i.e., $G/H$ admits the structure of a real vector bundle over $K/L$. Thus 
$$
d_{G/H}=\dim(G/H)-\dim(K/L)
$$
\end{theorem}

Since $G=K\times\mathbb R^s$ and $H=L\times\mathbb R^t$ where $L\subseteq K$
and $G/H\xrightarrow{\mathbb R^{s-t}}K/L$, we have $d_{G/H}=s-t$

For example since $X=\mathbb{C}^*$ is diffeomorphic to $\mathbb{S}^1\times\mathbb{R}$ we see that $d_X=1$.
Here we bring a more interesting example which we use later.

\begin{example}[Lemmas 1 and 2, \cite{Akh83}]\label{calculatingdX}  
Suppose $G$ is an algebraic group and $H$ is an algebraic subgroup with unipotent radicals $U$ and $V$ respectively. 
Then a direct computation of dimensions shows us that        
\[  
       d_{G/H}\; = \; \dim_{\mathbb C} G - \dim_{\mathbb C} H + \dim_{\mathbb C} U - \dim_{\mathbb C} V.
\]   
As a consequence, if $J$ is an algebraic subgroup such that $H\subset J \subset G$, then 
\[  
     d_{G/H} = d_{G/J} + d_{J/H}.
\]   
\end{example}

\begin{example}  \label{subgpBorel}   
We consider what happens for $S := SL(2,\mathbb{C})$ and various algebraic subgroups of this $S$.  
Note that $K =SU(2)$ here.    

\medskip   
{\bf Subexample (a):} \  
Let $H = B$ be a Borel subgroup of $S$.   
Then $S/H$ is compact and $d_{S/H} = 0$.   

\medskip   
{\bf Subexample (b):} \  
Let $H$ be a maximal unipotent subgroup of a Borel subgroup of $S$.  
Such $H$ is isomorphic to $\mathbb C$ and $L = \{ e \}$.   
The minimal compact orbit in $S/H$ is $K/L=SU(2)/\{e\}$ and $S/H$ fibers as a holomorphic   
$\mathbb C^*$--bundle over $\mathbb P_{1}$.   
It follows that $d_{S/H}=1$. 

\medskip   
{\bf Subexample (c):} \  
Now let $H = {\rm diag} (\alpha,\alpha^{-1})\cong\mathbb C^*$.  
Here $L \cong S^1$ and $V \cong    \mathbb R$, 
where the equivalences are topological.
Hence, the minimal compact orbit in $S/H$ is $K/L=SU(2)/{\mathbb S^1}$ and $S/H$ fibers 
as an affine $\mathbb C$--bundle over $\mathbb P_1$, i.e., $S/H$ is the affine quadric.   
Thus, $d_{S/H}=2$.   
The normalizer $N(H)$ of $H$ gives a 2:1 covering $S/H \to S/N(H)$, where the base 
is $\mathbb P_2 \setminus Q$ with $Q$ a quadric curve.  Thus $d_{S/N(H)}= d_{\mathbb P_2 \setminus Q} = 2$.  

{\bf Subexample (d):} \
Let $H = \{e\}$, then $d_S = 3$.  

\end{example}

\medskip   
\begin{example} 
Let $G$ be the Heisenberg group consisting of 
matrices of the form 
$$
     \left( \begin{array}{ccc}
      1 & x &y  \\
      0 & 1& z  \\
      0 & 0 & 1
     \end{array} \right)
 $$
where $x , y, z\in\mathbb C$. 
For any integer $k$ from 0 to 6, we can make discrete subgroups 
$\Gamma$ such that $d_{G/\Gamma}= k$. For example if we define $\Gamma$ to be a subgroup of $G$ with 
$x, z\in \mathbb Z$ and $y\in \mathbb{Z}+i\mathbb{Z}$
then $d_{G/\Gamma}=2$
since $G/\Gamma\xrightarrow{T}\mathbb{C}^*\times\mathbb{C}^*$, where $T$ is
a one dimensional complex torus.
\end{example}

%%%%%%%%%%%%%%%%%%   
%%%%%%%%%%%%%%%%%%   

\subsection{The fibration lemma}
If $H$ is a connected subgroup of a connected Lie group $G$ such that $I\subset H\subset G$, 
where $I$ is another connected subgroup of $G$,
by Iwasawa decomposition we have the homeomorphisms $G=K\times\mathbb R^s$, $H=L\times\mathbb R^t$, 
and $I=M\times\mathbb R^u$
where we can choose maximal compact subgroups $K, L$ and $M$ of $G , H$ and $I$ respectively,
such that $M\subseteq L\subseteq K$.    
By Theorem \ref{MOSTOW} we have the simple calculation    
$$
        d_{G/I} = s - u = s - t + t - u = d_{G/H} + d_{H/I}  .
$$
This was noted in example \ref{calculatingdX} in the setting of algebraic groups.

\medskip  
Since we are dealing with isotropy subgroups that are not connected, we need 
the next observations to handle these settings.   
Interesting such groups exist, e.g., see the proof of Proposition \ref{torbdle}.    

\begin{lemma}\label{flemma}      
Given a locally trivial fiber bundle $ X \stackrel{F}{\to} B$ with $F,X, B$ connected smooth manifolds 
the following were proved using spectral sequences in \S 2 in \cite{AG94}:    
\begin{enumerate}  
\item If the bundle is orientable (e.g., if $\pi_1(B) = 0$), then $d_X = d_F + d_B$.    
\item If $B$ has the homotopy type of a $q$--dimensional CW complex, then 
$d_X \ge d_F + (\dim B - q)$.    
\item If $B$ is homotopy equivalent to a compact manifold, then $d_X \ge d_F + d_B$.  
For example, this happens if $B$ is homogeneous with connected isotropy or is a solvmanifold \cite{Aus63}.   
\end{enumerate}  
\end{lemma}

%%%%%%%%%%%%%%%%%%   
%%%%%%%%%%%%%%%%%%   

\chapter{Main Results}  \label{mainresults}

In this chapter we give the classification theorem for homogeneous K\"{a}hler manifolds
with discrete isotropy and with top non--vanishing homology in codimension at most two.

In section \ref{solvable} we prove that all such manifolds are solvmanifolds. 
This partially answers the 
modified question of Akhiezer discussed in section \ref{Akhiezerquestion}.

In section \ref{Tstars} we prove that any homogeneous K\"{a}hler manifold with discrete isotropy which is a torus 
bundle over $\mathbb C^*\times\mathbb C^*$ splits as a product.    

We use this in section \ref{maintheorem} to prove the main theorem. 
We prove that any homogeneous K\"{a}hler manifold $X$ with $d_X = 2$ and with discrete isotropy 
is a product of a Cousin group and 
one of $\{ e \}, \mathbb C^*$, $\mathbb C$, or $(\mathbb C^*)^2$.

%%%%%%     

\section{Reduction to solvable case}  \label{solvable}

We now prove one of the key ingredients that will allow us to prove our classification in the case of discrete isotropy.    
First we need the following lemmas.  

\begin{lemma}\label{tbdloverC}
Let $G$ be a connected complex Lie group that acts holomorphically, transitively, and     
effectively on any of the following complex manifolds:
\begin{enumerate}  
\item $C\times \mathbb C$,  
\item $C\times\mathbb C^*$, or
\item $C\times(\mathbb C^*)^2$
\end{enumerate}  
where $C$ is a Cousin group.   
Then $G$ is solvable.
\end{lemma}

\begin{proof}   
We first claim that the $G$--action on $C\times A$ 
induces a transitive $G$-action on $A$, where $A$ is $\mathbb C$, $\mathbb C^*$, or $\mathbb C^* \times \mathbb C^*$.
To see this let $p: C\times A \to A$ be the projection map to the second factor.
Since $\mathscr{O}(C)\cong \mathbb C$ and $A$ is holomorphically separable,
$p(g.(C \times \{ z \}))$ is always a point in $A$, for all $g\in G$ and $z\in A$. 
Since the $G$--action is holomorphic and $p$ is also holomorphic, this induces the required action on $A$. 

\medskip    
Let $G=S\ltimes R$ be a Levi-Malcev decomposition of the complex Lie group $G$ (Theorem \ref{Levimalcev}). 
For any complex Lie group $H$, if $\phi: S\to H$ is a Lie homomorphism, then $\phi(S)$ is semi-simple 
(see section 5.2., \cite{Hum72}), a fact which we use later.

\medskip 
If $A= \mathbb C$, it follows that the $G$--action on $\mathbb C$ is given by 
a homomorphism $\psi:G\to {\rm Aut} (\mathbb C)$. 
But ${\rm Aut} (\mathbb C)=\mathbb C^*\ltimes\mathbb C$ is solvable, i.e., $\psi(S)=\{e\}$ by a fact mentioned above.

\medskip 
Let $A=\mathbb C^*$. 
Note that ${\rm Aut}(\mathbb C^*)^0 \cong \mathbb C^*$, which is solvable.
Hence $\psi(S)=\{e\}$ for any homomorphism $\psi:G\to {\rm Aut} (\mathbb C^*)$.

\medskip 
Let $A = (\mathbb C^*)^2$ and assume that the $G$--action on $(\mathbb C^*)^2$ is given by 
a representation $\phi : G \to {\rm Aut} (\mathbb C^*)^2$.
First, $\phi(S)$ has no 1-dimensional orbit in $A$.   
Such an orbit would be a $\mathbb P_1$.
Since $(\mathbb C^*)^2$ is holomorphically separable, this is not possible.
If $S$ has a two dimensional orbit,  
then the $S$--action is transitive, i.e., $(\mathbb C^*)^2 = S/H$ for some algebraic subgroup $H$ of $S$.      
Consider part of the exact homology sequence for the fibration $S \xrightarrow{H} S/H$
as follows    
\[
       \cdots \; \longrightarrow \; \pi_1(S) \; \longrightarrow \; \pi_1((\mathbb C^*)^2)   
       \; \longrightarrow \;  \pi_0 (H) \to \cdots.
\]
Both $\pi_1(S)$ and $\pi_0 (H)$ are finite, but $\pi_1((\mathbb C^*)^2)\cong\mathbb Z^2$ is infinite giving us 
a contradiction.   
Hence $S$ acts ineffectively on $A = (\mathbb C^*)^2$ too.  

\medskip    
Under the $G$ action on $C\times A$, the semisimple group $S$  
stabilizes $C \times\{z_0\}$, which is given by 
a representation $\mu: G\to {\rm Aut}^{0} (C) \cong C$.
This means $\mu(S)=\{e\}$. 
So $S$ is acting ineffectively, i.e., $S=\{e\}$.
Hence $G$ is solvable in all cases. 
\end{proof}

\begin{lemma}\label{dlem2}   
Let $G$ be a connected complex Lie group and $\Gamma$ a discrete subgroup of $G$ 
such that $X=G/\Gamma$ is a homogeneous K\"{a}hler manifold.   
Assume that $\Gamma$ is not contained in any proper parabolic subgroup of $G$,  
$\mathscr{O}(G/\Gamma) = \mathbb C$, and $d_X \leq 2$.   
Then $G$ is solvable. 
\end{lemma}    

\begin{proof}  
According to Lemma \ref{radicalclosed} the $R$ orbits are closed and we have the fibration:
\[   
            G/\Gamma \; \longrightarrow \;  G/R.\Gamma \; = \; S/\Lambda
\]  
where $\Lambda=S\cap R.\Gamma$ is Zariski dense and discrete in $S$.     
There are two cases. 
One case is when the $R$ orbits have no non-constant holomorphic functions. 
Using Theorem \ref{Gi04} we get the desired conclusion.
The other case is when $\mathscr{O} (R.\Gamma/\Gamma)\neq \mathbb C$. 
Let $J$ be the closed complex subgroup of $R .\Gamma$ that contains $\Gamma$ 
so that we have the holomorphic reduction $R.\Gamma/\Gamma \to R.\Gamma/J$.   
(Note that this exists by the discussion in \S \ref{hreduction}.)  
This holomorphic reduction  
gives the intermediate fibration 
\[
          G/\Gamma\to G/J\to G/R.\Gamma
\]  
 
Note that since $J^{0}\lhd G$ (Theorem \ref{HO84}, page \pageref{HO84}) we can make the following quotients: 
\[
      \widehat{G}=G/J^{0} \quad\quad \widehat{R}=R/J^{0} \quad\quad  \widehat{J}=J/J^{0}
\]
which gives $G/J=\widehat{G}/\widehat{J}$ and $\widehat{G}=S.\widehat{R}$, where $S$ is a Levi factor for $G$.
Applying the result stated in section \ref{flemma} 
on page \pageref{flemma} and the topological condition $d_{G/\Gamma}\leq 2$ we have
\[
        2 \; \geq \; d_{G/\Gamma} \; \geq \; d_{R.\Gamma/J} \; \geq \; \dim_{\mathbb C} R.\Gamma/J. 
          \;\;   
\]
Hence $\dim_{\mathbb C} \widehat{R}=1$, or 2. 
If $\dim_{\mathbb C} \widehat{R}=1$, then $S$ acts trivially on $\widehat{R}$ and 
Theorem \ref{Gi04} on page \pageref{Gi04} applies and    
$\widehat{G}$ is solvable.       
Since $\widehat{G} = S.\widehat{R}$ we have $S = \{e\}$.
But $G= S.R$, hence $G$ is also solvable.  
If $\dim_{\mathbb C}\widehat{R}=2$, then either $\widehat{G}=S\times \widehat{R}$ which again is 
the case dealt with in Theorem \ref{Gi04},
or $\widehat{G}$ is not a trivial product of $S$ and $\widehat{R}$. 
This with Theorem \ref{AG942} on page \pageref{AG942} proves $G$ is solvable.    
\end{proof}

%%%%%%%%   

\begin{lemma}\label{closedness}
Let $C$ be a connected Abelian Lie group and $A$ a closed, connected subgroup. 
Then $d_{A}\leq d_{C}$.
\end{lemma}

\begin{proof}
 The maximal compact subgroups in this setting are unique and 
$C=K\times V$ and $A=L\times W$ with $L\subseteq K$ 
and with the vector space group $W$ contained in the vector space group $V$. 
It now follows from the definition of $d$ that 
$ d_A = \dim_{\mathbb R} W \le \dim_{\mathbb R} V = d_{C}$. 
\end{proof} 
    
\begin{lemma} \label{Sorbitclosed}  
Let $G$ be a connected complex Lie group, and $\Gamma$ a discrete subgroup.
Let $X=G/\Gamma$ be a K\"{a}hler manifold.
For any semisimple Lie subgroup $S$ of $G$, the $S$ orbit $S/S\cap\Gamma$ is closed.
\end{lemma}

\begin{proof}
Let $S$ be a complex semi-simple Lie subgroup of $G$.   
Pick $x\in X$ so that the $S$ orbit through $x$ is of minimal dimension.
Such an $S$ orbit is necessarily closed inside $X$. Otherwise, it has a boundary 
of a smaller dimension (Theorem 3.6, \cite{GMO11}).  
Note that remark 2.2 in the same paper guarantees that we are in the situation of the theorem.   
Since $\dim S.x=\dim S-\dim (S\cap \Gamma)$, and $\dim (S\cap\Gamma)=0$, 
we conclude that all $S$ orbits are minimal, i.e. of the same dimension and closed. 
\end{proof}

\begin{lemma}\label{dthree}
Let $\Gamma$ be a finite subgroup of a semisimple Lie group $S$.
Then $d_{S/\Gamma} \geq 3$.
\end{lemma}

\begin{proof}
This follows since 
\[
      d_{S/\Gamma} \; = \; d_S \; = \; {\rm codim}_{\mathbb R} K \geq 3,
\]
where $S=K^{\mathbb C}$ and $K$ is any maximal compact subgroup of $S$.
\end{proof}

\begin{lemma}  \label{StimesR}
Suppose $G = SL(2,\mathbb C) \ltimes R$  and $\Gamma$ is a 
discrete subgroup $G$ with $X=G/\Gamma$ K\"{a}hler and $d_{X} \le 2$.   
Then $\Gamma$ is not contained in a proper parabolic subgroup $P$ of $G$
\end{lemma}  

\begin{proof} 
We assume $\Gamma$ is contained in a proper parabolic subgroup $P$ of $G$     
and derive a contradiction.    
Such a group has the form $P = B \ltimes R$,  
where $B$ is isomorphic to a Borel subgroup of $SL(2,\mathbb C)$.   
Consider the holomorphic reduction
\[  
       P/\Gamma \; \xrightarrow{J/\Gamma} \; P/J .   
\]    
and note that because $P$ is solvable, $J/\Gamma$ is a Cousin group \cite{OR88}  
and $P/J$ is a Stein manifold \cite{HO86}.  
Note that $\mathscr{O}(P/\Gamma) \not= \mathbb C$.  
Otherwise, $P/\Gamma$ itself would be a Cousin group and $P$ would be Abelian, a contradiction.         
Thus, $P/J$ is a positive dimensional Stein manifold with $d_{P/J} \leq 2$. 
Hence $P/J \cong \mathbb C$, $\mathbb C^*$, $\mathbb C^*\times \mathbb C^*$, or the complex Klein bottle.   

\medskip
Consider the following diagram with its induced $S=SL(2,\mathbb C)$-orbit fibrations    

\begin{equation}\label{semisimplefibration}
\begin{array}{ccccc}
      G/\Gamma \; &\;  \stackrel{F=J/\Gamma}{\longrightarrow} &  G/J       
      & \stackrel{P/J}{\longrightarrow}     & G/P=\mathbb{P}_1\\
       \cup    &                                    & \cup      &                                 &    ||          \\
      S/S\cap\Gamma \; & \;  \stackrel{A_1}{\longrightarrow}        & 
     S/S\cap J & \stackrel{A_2}{\longrightarrow}     & S/B=\mathbb P_1
\end{array}
\end{equation}  

\medskip   
Note that since $\dim P/J >0$ and $P/J$ is not compact, it follows that $d_F \le 1$ by the Fibration Lemma.  
Since $A_1$ is a closed Abelian subgroup of $F$, Lemma \ref{closedness} implies $d_{A_1}\leq 1$. 
Also Lemma \ref{dthree} for $S:=SL(2,\mathbb C)$ implies that $d_{S/S\cap\Gamma} = 3$.   
We will consider all the possibilities for $S \cap J$ and use the fibration $S/S\cap\Gamma \to S/S \cap J$  
in order to derive a contradiction.   

\medskip
\begin{itemize}
\item First observe that $G/J$ is K\"{a}hler in all cases:  

\subitem (a) The one dimensional space $P/J$ is equivariantly embeddable in $\mathbb P_1$.    
So $G/J$ is an open subset in a homogeneous $\mathbb P_1$--bundle over $\mathbb P_1$   
which is K\"{a}hler by a theorem of Kodaira \cite{Kod54}. 
Thus $G/J$ is K\"{a}hler.

\subitem (b) If $P/J\cong\mathbb C^*\times\mathbb C^*$, then Lemma \ref{flemma} gives $d_{G/J}=2$, and $d_F=0$.
In fact $F$ is a compact complex torus (see Lemma \ref{torus}). 
Since $F$ is a torus, by means of integration over this compact fiber \cite{Bla56}    
we can push the K\"{a}hler metric down to conclude $G/J$ is K\"{a}hler. 
\end{itemize}

Now since $G/J$ is K\"{a}hler, any S-orbit in $G/J$ is K\"{a}hler 
and, as a consequence,  has algebraic isotropy $S\cap J$ (Theorem \ref{BO})
The algebraic subgroups of a Borel subgroup $B$ in $S = SL(2,\mathbb C)$ are, up to isomorphism,  
finite subgroups,  $A=diag(\alpha,\alpha^{-1})$,  $N(A)$, a maximal unipotent subgroup $U$ of $B$, 
and $B$ itself.         

\begin{enumerate}
\item $\dim (S \cap J) = 2$. Then $S \cap J = B$, so $d_{S/B}=0$ (see example \ref{subgpBorel} (a)).
By Lemma \ref{flemma} we have $d_{S/S \cap \Gamma}=d_{A_1}+d_{S/B} \le 1+ 0 = 1 < 3$.  

\item $\dim (S \cap J) = 1$. We consider the fibration \ref{semisimplefibration} mentioned above.   
There are two possibilities:    
\subitem (i) $S\cap J\cong \mathbb C^*$. 
Then $S/S\cap J$ is an affine quadric or the complement of a quadric in $\mathbb P_2$.
In both cases $A_2 \cong \mathbb C$.    
Note that $P/J\neq \mathbb C^*$.
So $P/J \cong \mathbb C$ or $\mathbb C^*\times\mathbb C^*$.  
By fibration lemma (Lemma \ref{flemma}) we have $d_F = 0$. Hence $F$ is compact.    
By Lemma \ref{closedness} we see that $d_{A_1} = 0$. 
Thus $A_1$ is a torus.
But this also means $S\cap \Gamma$ is infinite which is in contradiction
with the fact that discrete isotropy group of K\"{a}hler manifolds transitive under a semisimple group must be finite
(see the remark after Theorem \ref{BO} on page \pageref{BO}).
    
\subitem (ii) $S\cap J\cong \mathbb C$. Then $S/S\cap J$ is a finite quotient of $\mathbb C^2 \setminus \{(0,0)\}$ and so
$A_2 \cong \mathbb C^*$. 
Then $P/J\cong\mathbb C$, $\mathbb C^*\times\mathbb C^*$, or $\mathbb C^*$.
For the first two cases $F$ is compact and we get the same contradiction as in (i).
If $P/J\cong\mathbb C^*$, then $d_F = 1$ by fibration lemma.   
Hence $d_{A_1}\leq 1$ and 
$A_1 \cong \mathbb C^*$, or a torus. In both cases $S\cap\Gamma$ is infinite which is in contradiction 
with the fact that $S\cap\Gamma$ is finite, as explained in the previous case.   
    
\item $\dim (S \cap J) = 0$ and then $S \cap J$ is finite.   
In this case $\dim (S/S \cap J) = \dim S - \dim (S \cap J) = 3$, 
and $\dim G/J = 3$. It then follows that $S/S \cap J$ is open in $G/J$. 
Note that the only case that $\dim G/J=3$ is when it is 
$\mathbb C^*\times\mathbb C^*$ bundle over $S/P$, i.e., $d_{G/J}=2$. 
Lemma \ref{Sorbitclosed} says this S-orbit is also closed. 
Thus, $G/J=S/S\cap J$ which means $d_{S/S \cap J}=2$. 
Since $d_{A_1}\leq d_F = 0$, we see that $d_{S/S\cap \Gamma}<3$, which is a contradiction.  
\end{enumerate}

We conclude that the assumption that there is a proper parabolic subgroup of $G$   
which contains $\Gamma$ yields a contradiction and the lemma follows.   
\end{proof}

We now combine the above to prove the main result of this section, 
namely $G$ is solvable in our setting.

\begin{proposition}\label{solvmanifold}   
Let $G$ be a connected complex Lie group and $\Gamma$ a discrete subgroup of $G$ 
such that $X=G/\Gamma$ is a homogeneous K\"{a}hler manifold.   
If $d_X \leq 2$, then $G$ is solvable. 
\end{proposition}    

\begin{proof}  
First suppose 
$X=S/\Gamma$ is a K\"{a}hler manifold, where $S$ is a semisimple complex Lie group.     
By Theorem \ref{BO} we see that $\Gamma$ is algebraic and so finite.
But then $d_X\geq 3$ by lemma \ref{dthree}. 
This is a contradiction to the assumption $d_X \leq 2$.
We conclude that $G$ cannot be semisimple and so    
must have a positive dimensional radical, i.e., $\dim R > 0$.

\medskip
We prove that $G$ is solvable by induction on the dimension of $G$.
Clearly, if  $\dim_{\mathbb C} G < 3$, then the group $G$ is solvable.    
So we assume that the result holds for any connected complex Lie group whose dimension is strictly
less than $n$ and assume $G$ is a complex Lie group with $\dim_{\mathbb C}G = n$.

\medskip 
We first claim that $\Gamma$ is not contained in a proper parabolic subgroup of $G$.    
In order to prove this we assume that $\Gamma$ 
is contained in a proper parabolic subgroup of $G$ and derive a contradiction.
Let $P$ be a maximal such subgroup.  
Since $P/\Gamma$ is K\"{a}hler $d_{P/\Gamma}\leq 2$, and $\dim P/\Gamma < n$,  
the induction hypothesis implies $P$ is solvable.      

\medskip   
Recall that a maximal proper parabolic subgroup of a complex semisimple Lie group $S$ is solvable if and only if it is 
isomorphic to the Borel subgroup $B$ in $S \cong SL(2,\mathbb C)$.  
Then $P=B\ltimes R \subseteq G = S\ltimes R$. 
We are now in the situation of lemma \ref{StimesR} which gives us the desired contradiction.
We conclude that $\Gamma$ is not contained in a proper parabolic subgroup of $G$.    

\medskip  
We now have to consider two cases as follows.        
First assume that $X$ has no non-constant holomorphic functions, i.e., $\mathscr{O}(X)=\mathbb C$.
Then lemma \ref{dlem2} shows that $G$ is solvable.  

\medskip  
Next assume $X$ has non-constant holomorphic functions, i.e., $\mathscr{O}(X)\neq\mathbb C$.  
Under these assumptions the classification given in the Main Theorem in \cite{AG94} 
(see also Theorem \ref{AG941} on page \pageref{AG941}) applies and 
the base $G/J$ of the holomorphic reduction 
\[
     G/\Gamma \xrightarrow{J/\Gamma}  G/J
\]
is one of the following:    
\begin{enumerate}  
\item an affine cone minus its vertex, 
\item the affine line $\mathbb C$,
\item the affine quadric $Q_2$,
\item ${\mathbb P}_2 \setminus Q$ where $Q$ is a quadric curve, or
\item a homogeneous holomorphic $\mathbb C^*$-bundle over an affine cone with its 
vertex removed.
\end{enumerate}   
Note that in case 2 the fiber is a torus. By Lemma \ref{tbdloverC}, $G$ is solvable.

For 3 $\&$ 4 we get fibrations
\[
G/\Gamma \to G/J \to G/P \cong \mathbb P_1
\]
with $P$ a proper parabolic subgroup of $G$.
Now $\Gamma \subseteq P$ gives a contradiction.

For 1 $\&$ 5 first note that a cone minus its vertex is a $\mathbb C^*$ bundle 
over a flag manifold $G/P$ (see section \ref{attemp}),
where $P$ is a parabolic subgroup in $G$.
In both cases we get the fibration
\[
G/\Gamma \xrightarrow{J/\Gamma} G/J \longrightarrow G/P.
\]
We cannot have 
$G\neq P$, i.e., $P$ a proper parabolic subgroup, because then $\Gamma \subset P \subset G$ 
which is a contradiction, as noted above.      
If $G = P$, i.e., if the base of the second fibration is a point, then 
$G/J$ or a 2:1 covering of it is biholomorphic to $\mathbb C^*$ or $(\mathbb C^*)^2$.
Note that if $G/J$ or its 2:1 covering is  $(\mathbb C^*)^2$, the fiber $J/\Gamma$ is a torus.
By Lemma \ref{tbdloverC} it follows that $G$ is solvable.   
If $G/J = \mathbb C^*$, then the fiber $J/\Gamma$ is K\"{a}hler and 
there are two separate cases to consider.   
In 5 this fiber will be compact and we have a torus bundle over $\mathbb C^*\times\mathbb C^*$.  
The result then follows from Lemma \ref{tbdloverC}.  
In 1 if $d_X =2$, then $d_{J/\Gamma} = 1$ by the fibration lemma.  
Since $J$ has dimension strictly less than the dimension of $G$,  
it follows by the induction hypothesis that $J$ is solvable.   
But then $G$ is solvable too, since $G/J = \mathbb C^*$,        
so completing the proof.      
\end{proof}   

\section{Triviality of torus bundle over $\mathbb C^*\times\mathbb C^*$}\label{Tstars}

In subsection \ref{non-compact} we consider a discrete subgroup $\Gamma$ of the 3--dimensional 
Heisenberg group $G$ such that the holomorphic reduction of $G/\Gamma$ is a non--trivial   
holomorphic torus bundle over $\mathbb C^*\times\mathbb C^*$.    
The reason the bundle is not even homeomorphic to a product $T \times (\mathbb C^*)^2$ is  because 
the fundamental group of the latter is Abelian, while the fundamental group $\pi_1(G/\Gamma) \cong \Gamma$ 
is nilpotent and not Abelian.    
Note that this example is not K\"{a}hler.   
In sharp contrast to this we consider in this section 
homogeneous K\"{a}hler solvmanifolds with discrete isotropy which fiber       
as torus bundles over $\mathbb C^*\times\mathbb C^*$ and  prove that 
such $X$ are biholomorphic to the product of the torus and  $(\mathbb C^*)^2$.  
The structure of the holomorphic reductions of K\"{a}hler solvmanifolds 
has been analyzed in \cite{OR88} and \cite{GO08}, see also Theorem \ref{OeRi88} and Theorem \ref{GO08}.   
   
\begin{proposition}   \label{torbdle}  
Suppose $G$ is a connected, 
simply connected solvable complex Lie group and $\Gamma$ is a discrete subgroup 
such that $G/\Gamma$ is K\"{a}hler and has holomorphic reduction 
\[  
    X \; = \;  G/\Gamma \; \stackrel{T}{\longrightarrow} \; G/J 
           \; \cong \; \mathbb C^*\times\mathbb C^* \; = \; Y  ,     
\]  
where $T = J/\Gamma$ is a compact complex torus.   
Then a finite covering of $G/\Gamma$ is biholomorphic to    
the product $T \times \mathbb C^*\times\mathbb C^*$.    
\end{proposition}

\begin{proof} 
There are two cases to consider, namely either $J^0$ is normal in $G$ or not. 
If $J^0$ is normal in $G$, then the group $G/J^0$ is a simply connected 2--dimensional complex Lie group.  
Up to isomorphism there are two possibilities, the Abelian case and the solvable case.  

\medskip   
First assume that the group $G/J^0$ is Abelian and let $\alpha:G \to G/J^0$ be the projection homomorphism.        
Since $Y$ contains the real subgroup $\mathbb S^1 \times \mathbb S^1$
the pullback $G_0 := \alpha^{-1} (\mathbb S^1 \times \mathbb S^1)$ is a real subgroup of $G$.
We then have the fibration
\begin{equation*}
\begin{array}{ccc}
      G/\Gamma \; &\;  \stackrel{T}{\longrightarrow}   &  \mathbb C^*\times\mathbb C^* \\
       \cup       &                        & \cup               \\
 G_0 /\Gamma\ \; & \; \stackrel{T}{\longrightarrow}  & \mathbb S^1 \times \mathbb S^1
\end{array}
\end{equation*}
with compact fibers $T$. Thus $G_0 / \Gamma$ is compact.
The triple $(G, G_0, \Gamma)$ is a K\"{a}hler CRS manifold (see subsection \ref{CRSmanifolds} for definition).
Thus the bundle splits as a product (Theorem 6.14, 2(iii) page 189, \cite{GO11}).

\medskip  
Next we assume that $G/J^0$ is isomorphic to $B$, the Borel subgroup of $SL(2,\mathbb C)$, and   
we let $\mathfrak b$ denote its Lie algebra.     
Let $\mathfrak g$ be the Lie algebra of $G$ and $\mathfrak{g/j}$ be the Lie algebra of $G/J^0$.   
Let $\pi : G\to G / J^0$ be the quotient map and $d\pi : \mathfrak g \to  \mathfrak {b:=g/j}$ its differential.

\medskip
By definition the nilradical $\mathfrak n$ is the largest nilpotent ideal in $\mathfrak g$.
Since the sum of two nilpotent ideal is nilpotent (Proposition 6 on p.25 in \cite{Jac62}),     
$\mathfrak n$ contains every nilpotent ideal in $\mathfrak g$.
In particular, $\mathfrak g'$ and $\mathfrak j$ are in $\mathfrak n$.
Also $d\pi(\mathfrak n)$ is a nilpotent ideal in $\mathfrak{g/j}= \mathfrak b$.
Hence, it is contained in the nilradical $\mathfrak n_b$ of $\mathfrak b$.
Thus, we have $\mathfrak j\subseteq \mathfrak n \subseteq d\mathfrak \pi^{-1}(\mathfrak n_b)$.
Note that $\dim\mathfrak j = n-2$ and $\dim d\pi^{-1}(\mathfrak n_b)=n-1$.
But $\mathfrak n\neq\mathfrak j$ since $\mathfrak{g/j}$ is not Abelian (the quotient by the nilradical is always Abelian).
So $\mathfrak n = d\pi^{-1}(\mathfrak n_b)$.   
Let $N$ be its Lie group. 
Define $\Gamma_N := \Gamma\cap N$.   
Since $N$ is nilpotent, the exponential map $\exp: \mathfrak n \to N$ is surjective. 
Let $\gamma\in\Gamma_N$.    
Then there is $y\in\mathfrak n$ such that $\gamma=\exp (y)$.
Let $\mathfrak u = <y>_{\mathbb C}$ be the Lie algebra generated by $y$.   
Let $U$ be its Lie group. 
Then $\mathfrak n = \mathfrak u \oplus \mathfrak j$ and so $N=U\times J^0$.
Note that $[y,\mathfrak j]=0$, since $\Gamma$ centralizes $J^0$, see the proof of Theorem 1 in \cite{GO08}.         
So $N$ is Abelian.
Define $\Gamma_U : =\Gamma\cap U \cong\mathbb Z$ and $\Gamma_J := \Gamma \cap J^0$ 
which is a full lattice in $J^0$.  
It is immediate that    
\[   
         N/{\Gamma_N} = U/{\Gamma_U} \times J^{0} / {\Gamma_J}.   
\]              

\medskip
Since $\Gamma/\Gamma_N \simeq \mathbb Z$, we choose $\gamma\in\Gamma$ so that 
$\gamma$ projects to a generator of $\Gamma/\Gamma_N$.  
Pick any $w\in \mathfrak{g} \setminus \mathfrak n$ and define $A:=\{\exp(sw):s\in\mathbb C\}$.    
Note that $A$ is complementary to $N$ in $G$, and thus $G = A \ltimes N$.
Let $\gamma_A \in A$ and $\gamma_N \in N$ so that 
$\gamma = \gamma_A . \gamma_N$.
Then since $\gamma$ and $ \gamma_N$ both centralize $J^0$, 
we see that $\gamma_A$ centralizes $J^0$. 
Since $\gamma_A\in A$,  
we have $\exp(h) = \gamma_A$ with $h = sw$ for some $s\in\mathbb C$.   
Note that    
\begin{equation}  \label{central}  
                     [h, \mathfrak j] \; = \; 0 , 
\end{equation}  
a fact which we use later.

\medskip
Since $\mathfrak a + \mathfrak u$ is isomorphic to 
$\mathfrak b = \mathfrak {g/j}$ as a vector space,   
we may take $e_+ \in \mathfrak u$ so that 
\[   
     [ d\pi (h) , d\pi ( e_+ ) ] \;  = \; 2 d\pi (e_+ )  .
\]   
Let $\{e_1 , \cdots, e_{n-2}\}$ be a basis for $\mathfrak j$. 
It follows that there are structure constants $a_i$ 
so that 
\[ 
         [h,e_+ ] \; = \; 2e_{+} \; +  \; \sum_{i=1}^{n-2} a_i e_i  .      
\]  
The remaining structure constants are 0, due to the fact $[h, \mathfrak j]=0$, see (\ref{central}).     
Conversely, any choice of the $a_i$ defines a solvable Lie algebra $\mathfrak g$ 
and the corresponding simply-connected group $G=A\ltimes N$.

\medskip
We compute the action of $\gamma_A\in A$ on $N$ by conjugation.   
For this note that the adjoint representation restricted to $\mathfrak n$, i.e., 
the map ${\rm ad}_h: \mathfrak n \to \mathfrak n$, is expressed by the matrix
\[   
   M \; := \; [{\rm ad}_h] \; = \;  
     \begin{pmatrix}
         2       & 0 & \cdots & 0 \\
         a_1     & 0 & \cdots & 0 \\
        \vdots  & \vdots  & \ddots & \vdots  \\
        a_{n-2} &  0      & \cdots & 0
      \end{pmatrix}.
 \]        
So the action of $A$ on $N$ is through the one parameter group of 
linear transformations $t\to e^{tM}$ for $t\in \mathbb C$.
For $k\geq 1$
$$
              (tM)^k  \; = \; \frac{1}{2} (2t)^k M,
$$
and it follows that     
\[   
      e^{tM} \; = \; \frac{1}{2}(e^{2t}-1)M \; + \; {\rm Id} .   
\]   
Now the projection of the element $\gamma_A$ 
acts trivially on the base $Y = G/J$, so $t=\pi i k$ where $k\in\mathbb Z$. 
Hence $\gamma_A$ acts trivially on U.
Also $\gamma_N$ acts trivially on $N$, since $N$ is Abelian.   
Thus $\gamma$ acts trivially on $N$ and as a consequence, 
although $G$ is a non-Abelian solvable group
the manifold $X=G/\Gamma$ is just the quotient of $\mathbb C^n$ by a discrete additive subgroup of rank $2n-2$. 
Its holomorphic reduction is the original torus bundle which, since we are now dealing with the Abelian case, is trivial.

\medskip  
Assume $J^0$ is not normal in $G$ and set $N := N_G(J^0)$. 
Let 
\[
    G/J \; \xrightarrow{N/J} \; G/N
\]
be the normalizer fibration.
Since the base $G/N$ of the normalizer fibration is an orbit 
in some projective space 
Lie's flag theorem applies (see section \ref{normalizerfibration})   
and $G/N$ is holomorphically separable and thus Stein \cite{Sno85}.         
Since we also have $d_{G/N}\leq 2$ we see that $G/N \cong \mathbb C$, $\mathbb C^*$ or $\mathbb C^*\times \mathbb C^*$.      
Assume $G/N \cong \mathbb C$.   
Since $d_X \leq 2$ 
the fibration lemma (Lemma \ref{flemma}) applies and $d_{N/J} = 0$, i.e., $N/J$ is biholomorphic to a torus $T$. 
By Grauert theorem $G/J \cong T\times \mathbb C$.
However, we assume that $G/J = \mathbb C^*\times\mathbb C^*$ and 
this gives us a contradiction.
Assume $G/N \cong \mathbb C^* \times \mathbb C^*$.
By Chevalley's theorem (Theorem 13, page 173, \cite{Che51}) the commutator group $G'$ acts algebraically.    
Hence the $G'$--orbits are closed and one gets the commutator fibration 
\[
        G/N \xrightarrow{\mathbb C} G/G'.N.
\]
Since $G$ is solvable, it follows that $G'$ is unipotent and the $G'$--orbits are cells, i.e.,  $G'.N/N \cong \mathbb C$.
By the Fibration Lemma the base of the commutator fibration is a torus. 
But it is proved in \cite{HO81}  
that the base of a commutator fibration is always Stein which is a contradiction.
Now the only case remaining is when $G/N \cong \mathbb C^*$. 
By the fibration lemma (Lemma \ref{flemma}) we get $d_{N/J} = 1$ and hence $N/J = \mathbb C^*$ 

\medskip    

Since $G$ is simply connected,       
$G$ admits a Hochschild-Mostow hull (\cite{HM64}), i.e., there exists a solvable linear algebraic group 
\[
       G_a \; = \; (\mathbb C^*)^k \ltimes G
\]
that contains $G$ as a Zariski dense, topologically closed, normal complex subgroup.   
By passing to a subgroup of finite index we may, without loss of generality,     
assume the Zariski closure $G_a (\Gamma)$ of $\Gamma$ in $G_a$ is connected.   
Gilligan and Oeljeklaus \cite{GO08} proved that $G_a (\Gamma) \supseteq J^0$, 
and that $G_a (\Gamma)$ is nilpotent.   
Let $\pi: \widehat{G_a} (\Gamma) \to G_a (\Gamma)$ be the universal covering and     
set $\widehat{\Gamma} := \pi^{-1}(\Gamma)$.       
Since $\widehat{G_a} (\Gamma)$ is a simply connected, nilpotent, complex Lie group, the exponential map from the 
Lie algebra ${\mathfrak g}_a (\Gamma)$ to $\widehat{G_a} (\Gamma)$ is bijective.   
For any subset of $\widehat{G_a} (\Gamma)$ and, in particular for the subgroup       
$\widehat{\Gamma}$, the smallest closed, connected, 
complex (resp. real) subgroup $\widehat{G_1}$ (resp. $\widehat{G_0}$) of $\widehat{G_a} (\Gamma)$
containing $\widehat{\Gamma}$ is well-defined.   
By construction $\widehat{G_0}/\widehat{\Gamma}$ is compact  -- 
e.g., see Theorem 2.1 (2) $\Longleftrightarrow$ (4) in Raghunathan (\cite{Rag08}).  
Set $G_1 :=\pi(\widehat{G_1})$ and $G_0 :=\pi(\widehat{G_0})$.   
We consider the CRS manifold $(G_1 , G_0 , \Gamma)$, see \S 2.4.3.   
Note that the homogeneous CR--manifold $G_0/\Gamma$ 
fibers as a $T$--bundle over $S^1\times S^1$.  
In order to understand the complex structure on the base $S^1\times S^1$   
of this fibration consider the following diagram     
\begin{equation*}
  \begin{array}{rcccc}  
  \widehat{G_0}/ \widehat{\Gamma} & \subset &  \widehat{G_1} / \widehat{\Gamma}  &   
  \subseteq  & \widehat{G_a}(\Gamma) / \widehat{\Gamma}  \\  
  || & & ||  &  &  ||   \\ 
   G_0/\Gamma & \subset &  G_1 / \Gamma  &   \subseteq  & G_a(\Gamma) /\Gamma \\  
   T \downarrow  & & T \downarrow & & \downarrow T \\   
   S^1 \times S^1 \; = \; G_0 / (G_0 \cap J^0 . \Gamma) & \subset & G_1 / J^0.\Gamma  & \subseteq  & G_a / J^0.\Gamma
 \end{array}
\end{equation*}   
As observed in the proof of Part I of Theorem 1 in \cite{GO08} the manifold $G_a / J^0.\Gamma$   
is a holomorphically separable solvmanifold and thus is Stein \cite{HO86}.  
So $G_1 / J^0 . \Gamma$ is also Stein and thus 
$G_0 / (G_0 \cap J^0 . \Gamma) $ is totally real in $G_1/J^0. \Gamma$.  
The corresponding complex orbit $G_1 / J^0.\Gamma $ is then 
biholomorphic to $\mathbb C^*\times\mathbb C^*$.  
It now follows by Theorem 6.14 in \cite{GO08} that a finite covering of $G_1/\Gamma$ splits as a product 
of a torus with $\mathbb C^*\times\mathbb C^*$ and, in particular, that 
a subgroup of finite index in $\Gamma$ is Abelian.  
 
\medskip    
Now set $A := \{ \ \exp\ t\xi \ | \ t \in\mathbb C \ \}$, where $\xi \in \mathfrak g \setminus \mathfrak n$ 
and $\mathfrak n$ is the Lie algebra of $N^0$.  
Then $G = A \ltimes N^0$ and any $\gamma\in\Gamma$ can be written as $\gamma = \gamma_A . \gamma_N$ 
with $\gamma_A\in A$ and $\gamma_N\in N$.   
The fiber $G/\Gamma \to G/N$ is the $N^0$-orbit of the neutral point and $\Gamma $
acts on it by conjugation.  
Since $N/\Gamma$ is K\"{a}hler and has two ends, it follows by Proposition 1 in 
\cite{GOR89} that (a finite covering of) $N/\Gamma$ is biholomorphic to a product  
of the torus and $\mathbb C^*$.   
(By abuse of language we still call the subgroup having finite index $\Gamma$.)    
Now the bundle $G/\Gamma \to G/N$ is associated to the bundle 
\[  
        \mathbb C \; = \; G/N^0 \; \longrightarrow \; G/N \; = \; \mathbb C^*     
\]  
and thus $G/\Gamma = \mathbb C \times_{\rho} (T \times \mathbb C^*)$,     
where $\rho : N/N^0 \to {\rm Aut } (T \times \mathbb C^*)$ is the adjoint representation. 
Since $\Gamma$ is Abelian, this implies $\gamma_A$ acts trivially on $\Gamma_N := \Gamma \cap N^0$.   
Now suppose $J$ has complex dimension $k$.   
Then $\gamma_A$ is acting as a linear map on $N^0 = \mathbb C \ltimes J^0 = \mathbb C^{k+1}$ and 
commutes with the additive subgroup $\Gamma_N$ that has rank $2k+1$ and spans $N^0$ as a linear space.        
This implies $\gamma_A$ acts trivially on $N^0$ and, as a consequence, the triviality of 
a finite covering of the bundle, as required.    
\end{proof}   

%%%%%%%%%%%%%%%%%%%%%%%%%%
%%%%%%%%%%%%%%%%%%%%%%%%%%

\section{Main theorem}\label{maintheorem}

In the following we classify K\"{a}hler $G/\Gamma$ when $\Gamma$ is discrete and $d_X \le 2$.   
Note that $d_X=0$ means $X$ is compact and this is the classical result of 
Borel--Remmert (Theorem \ref{torus.flag}) and the case $d_X = 1$ corresponds to $X$ having 
more than one end and this was handled in \cite{GOR89}    
(Theorem, p.164); the proof 
will show this equivalence.       

\begin{theorem}  \label{disc}
Let $G$ be a connected complex Lie group and $\Gamma$ a discrete 
subgroup of $G$ such that $X := G/\Gamma$ is K\"{a}hler and $d_X \le 2$.  
Then the group $G$ is solvable and a finite covering of $X$ is a direct product $C \times A$, where $C$ 
is a Cousin group and $A$ is $\{ e \}, \mathbb C^*$, $\mathbb C$, or $(\mathbb C^*)^2$.   
Moreover, $d_X = d_C + d_A$.    
\end{theorem}    

\begin{proof}
By proposition \ref{solvmanifold} we only need to consider the case that $G$ is solvable.
If there is no non-constant holomorphic function on $X$, i.e., 
if $\mathscr{O}(X)\cong\mathbb C$, then $X$ is a Cousin group (Theorem \ref{OeRi88} on page \pageref{OeRi88}).   
The case when $X$ is compact ($d_X =0$) is Theorem \ref{torus} on page \pageref{torus}.   

So we assume that $\mathscr{O}(X)\neq\mathbb C$, i.e., we have non-constant holomorphic functions on $X$.
Let
\begin{displaymath}
    \xymatrix{
          G/\Gamma \ar[r]^{C=J/\Gamma} &  G/J}
\end{displaymath} 
be the holomorphic reduction of $G/\Gamma$.       
It is known that the base is Stein (see the Theorem on p. 58 in \cite{HO86}), 
a finite covering of the bundle is principal (Theorem \ref{GO08} on page \pageref{GO08}),  
and the fiber is biholomorphic to a Cousin group (Theorem \ref{OeRi88} on page \pageref{OeRi88}).            
Since $G/J$ is Stein, by the fibration lemma (Lemma \ref{flemma} on page \pageref{flemma})    
one can see that 
\[   
       \dim_{\mathbb C} G/J \; \leq \; d_{G/J} \; \leq \; d_X\; \leq  \; 2  .     
\]       

Let $d_X=1$.  
Then $d_{G/J}=1$.   
Since $G/J$ is Stein, $G/J$ is biholomorphic to $\mathbb C^*$.
Since Theorem \ref{GO08} on page \pageref{GO08}   
states that a finite cover of a K\"{a}hler solvmanifold has 
a holomorphic reduction that is a Cousin group principal bundle,      
this means up to some finite cover the structure group of the bundle is a connected complex Lie group.
So the Grauert--R\"{o}hrl theorem (Theorem \ref{Grauert}) then states that this finite covering is trivial. 
   
Let $d_X=2$.   
Again by fibration lemma it follows that $G/J\cong \mathbb C$, 
$\mathbb C^*$, $\mathbb C^*\times \mathbb C^*$ or a complex Klein bottle \cite{AG94}.   
The case of $\mathbb C^*$ is handled as above.  
A torus bundle over $\mathbb C$ is trivial by Grauert--Oka principle (Theorem \ref{GOka}).
To be more specific note that since $\mathbb C$ is continuously contractible to a point, 
we conclude that the bundle is a topologically trivial bundle and     
hence is holomorphically trivial by Theorem \ref{GOka}. 
Finally, it is enough to discuss the case $\mathbb C^*\times\mathbb C^*$ since we discussed 
the first two cases already and    
a Klein bottle is a 2-1 cover of $\mathbb C^*\times\mathbb C^*$.       
But this is proposition \ref{torbdle}.    
The proof of the theorem is now complete as we covered all the possibilities.
\end{proof}

%%%%%%%%%%%%%%%%%%%%%%%%%

\begin{remark}    
We list the possibilities that can occur in the theorem.   
\begin{enumerate} 
\item   Suppose $d_{X} = 0$.   
Then $X$ is a torus.     
This corresponds to $X$ compact.      
\item  Suppose  $d_X = 1$. Then one of the following holds:         
\subitem (i) $X$ is a Cousin group of  hypersurface type     
\subitem (ii) a finite covering of $X$ is $T\times \mathbb C^*$ with $T$ a torus 
\newline   This corresponds to $X$ having  two ends, see \cite{GOR89}.      
\item   For $d_X=2$, then one of the following holds:        
\subitem (i) $X$ is a Cousin group  
\subitem (ii) a  finite covering of $X$ is a Cousin group of hypersurface type times $\mathbb C^*$, i.e,
of the form $C\times \mathbb C^*$, where $C$ is a Cousin group with $d_C = 1$.   
\subitem (iii) $X$ is a product $T\times\mathbb C$    
\subitem (iv) a finite covering of $X$ is a product $T\times(\mathbb C^*)^2$     
\end{enumerate} 
\end{remark}

%%%%%%%%%%%%%%%%%%%%%%%%%

\begin{remark} 
Let $G$ be the product of the 3--dimensional Heisenberg group and $\mathbb C$.  
As noted in Example 6 (a) in \cite{OR88} there is a discrete subgroup $\Gamma$ of $G$ 
such that $G/\Gamma$ is K\"{a}hler and $d_{G/\Gamma} = 3$.    
No finite covering of its holomorphic reduction splits as a product.  
So $d=2$ is optimal for the Main Theorem.   
\end{remark}   

%%%%%%%%%%%%%%%%%%%%%%%%%

\begin{remark}
The example discussed in subsection \ref{non-compact} has $d = 2$, but is not 
K\"{a}hler and its holomorphic reduction does not split holomorphically as a product.  
\end{remark}

%%%%%%%%%%%%%%%%%%%%%%%%%

\chapter{Future Work}

\section{Complete classification}
Let $G$ be a complex Lie group and $H$ a closed subgroup 
such that the homogeneous manifold $X=G/H$ is K\"{a}hler with $d_X \leq 2$.
In this dissertation we give a classification of such manifolds when the isotropy $H$ is discrete.
However a classification in the general case is missing.
We intend to solve the general case where $H$ is not necessarily discrete.  
In chapter \ref{tools} we introduced the normalizer and commutator fibrations.
By fibration lemma (Lemma \ref{flemma}) it is immediate that the base of the normalizer fibration 
\[
X:= G/H \to G/N
\]
has $d_{G/N}\leq 2$.

Chevalley showed (Theorem I3, page 173, \cite{Che51}) that $\mathfrak{g}' = \overline{\mathfrak{g}}'$ and 
thus $G' = \overline{G}'$ is acting as an {\bf algebraic group} on a projective space. 
Hence there is a closed $G'$ orbit. But we know that $G'\unlhd G$. Thus all $G'$ orbits are closed and we can 
consider the commutator fibration  
\[
G/N \xrightarrow{G'.N/H} G/G'.N
\]

The base of the commutator fibration is a positive dimensional Abelian Stein Lie group.
One has to analyze all the possibilities of the bundle over this Stein base. 

\section{Globalization of holomorphic action}

Let $X$ be a connected, compact, complex manifold.    
It was proved by Bochner--Montgomery \cite{BM47}   
proved that the identity component  $G=Aut^{0}(X)$ of the automorphism group of $X$
is a complex Lie group acting holomorphically on $X$.    
Assume $H$ is a connected real Lie group acting holomorphically on $X$.
Any $H$ orbit lies inside the corresponding $G$ orbit. 
In fact if $H$ acts transitively on $X$, i.e., if $H.x=X$, then $X=G.x$.
So a holomorphic transitive action of a real Lie group automatically extends to the complex Lie group $G$.   

\medskip
Gilligan and Heinzner extended this result to any connected homogeneous complex manifold $X$ with $d_X=1$ 
which can be thought of as "very close to compact" homogeneous manifolds \cite{GH98}.

\medskip
The situation fails for bigger $d_X$. 
An obvious example of this is an open disk in $\mathbb C$ which has $d=2$.

\medskip
For a more interesting example consider the action of $H=SL(3,\mathbb R)$ on $\mathbb C\mathbb P_2$.
There are two orbits which are $\mathbb R\mathbb P_2$ and 
its complement $X=\mathbb C\mathbb P_2\setminus\mathbb R\mathbb P_2$ which has $d_X=2$.
No complex Lie group acts transitively and holomorphically on $X$ \cite{HS81} 
Analyzing the real group action similar to the methods used in \cite{GH98} and using the classification of $X=G/H$
where $G$ is {\bf linear algebraic} homogeneous manifold with $d_X=2$ \cite{Akh83} and looking at 
the real orbits inside the homogeneous manifolds with the methods used in \cite{HS81}    
we have strong evidence that we can prove unless some 
special cases like the example above happen, 
any homogeneous manifold $X=G/H$ which is homogeneous under the holomorphic 
and transitive action of a {\bf Lie group}
$G$ with $d_X=2$ is also homogeneous under the holomorphic and transitive action of the 
globalization $G^{\mathbb C}$ of $G$. 

\medskip   
One of the essential tools we can use is an analogue of the normalizer fibration which was defined 
for a complex manifold $X$ homogeneous 
under the transitive holomorphic action of a real Lie group $G$ in the Nancy Band \cite{HO84}.    
The following observations concerning this fibration are essential for our purposes:    
\begin{enumerate}  
\item Given $X=G/H$ a homogeneous complex manifold, there exists a closed subgroup $J$ 
of $G$ containing $H$ with $J \subset N(H^{0})$ (but not necessarily equal!) such that 
$G/J$ is $G$-equivariantly an orbit in some projective space.   
\item  The orbit $G/J$ is open in the corresponding $G^{\mathbb C}$-orbit, where $G^{\mathbb C}$ 
denotes the smallest connected complex subgroup of the automorphism group of the projective 
space that contains $G$.     
\item The fiber $J/H$ is complex parallelizable.  
\item The bundle $G/H \to G/J$ is a locally trivial {\bf holomorphic} fiber bundle with 
structure group the {\bf complex}  Lie group $J/H^{0}$.  
\end{enumerate}  

\section{The Akhiezer Question} 

Almost thirty years ago, Akhiezer \cite{Akh84} posed a question concerning the existence of 
analytic hypersurfaces in complex homogeneous manifolds.  
The following variant of his question turns out to be essential for our proof of Theorem \ref{disc} 
stated above in order to show that the group $G$ is solvable.   
As far as we know, there is no general answer to this question.  
But there is enough known in special cases to be useful for our present purposes.  

\subsection{Modified Question of Akhiezer}   
Suppose $G/H$ is a K\"{a}hler homogeneous manifold that satisfies 
\begin{enumerate}  
\item [] (a) \; $\mathscr{O}(G/H) \simeq\mathbb C$  
\item [] (b) \; there is no proper parabolic subgroup of $G$ that contains $H$ 
\end{enumerate}  
Then $G/H$ is an Abelian complex Lie group; in particular, the group $G$ is solvable.  

In the future we intend to consider other settings where this problem might be solvable.    

\section{Stein sufficiency conditions}    

Suppose $G$ is a connected complex Lie group and $H$ is a closed complex subgroup    
such that $X := G/H$ is holomorphically separable.    
There are obvious examples where such $X$ are not Stein, e.g., $X = \mathbb C^n \setminus \{ 0 \}$    
for $n > 1$.    

\medskip    
One setting where one might hope that holomorphic separability implies Stein for complex    
homogeneous manifolds would be if the isotropy subgroup is discrete.     
There is, however, a counterexample to this in \cite{KO92}.    
In this counterexample one has a homogeneous complex manifold $Y := G/\Gamma$ 
with $d_{G/\Gamma} = 4 < 5 = \dim_{\mathbb C} G/\Gamma$.    
This violates the homological condition of Serre (see footnote, p. 79 in \cite{Ser53}) 
for Steinness of a complex manifold.    

\medskip   
This suggests that one should impose a further condition.   
Suppose $X = G/\Gamma$ is holomorphically separable and satisfies the condition    
\[  
       d_{G/\Gamma} \; \ge \;  \dim_{\mathbb C} G/\Gamma    .    
\]     
Again, one constructs counterexamples by considering $X := \mathbb C^k \times Y$, 
where $Y$ is the complex homogeneous manifold in the previous paragraph.  
As $k$ increases by one, the invariant $d$ increases by two and the dimension by one.   

\medskip    
Some further type of condition is needed that would prevent such examples.    
We will study this question in the setting where $G = S \ltimes R$ is a mixed group  
and $R$ is a simply connected nilpotent complex Lie group with the property that     
the smallest connected complex Lie subgroup of $G$ that contains $\Gamma\cap R$    
is assumed to be $R$.     
This can be expanded to include cases where $G$ is algebraic and    
$\Gamma\cap R$ is Zariski dense in $R$.

\end{document}